\documentclass[12pt]{amsart}
\usepackage{amssymb,amsmath, amsthm,latexsym,verbatim,stmaryrd}
\usepackage{a4wide}
\usepackage[hypertex]{hyperref}
\usepackage{enumitem}

\theoremstyle{plain}
\newtheorem{theorem}{Theorem}
\newtheorem{lem}{Lemma}[section]
\newtheorem{theo}[lem]{Theorem}
\newtheorem{prop}[lem]{Proposition}
\newtheorem{corollary}[lem]{Corollary}

\theoremstyle{definition}

\theoremstyle{remark}
\newtheorem {bmrk}[theorem] {Remark}

\newcommand{\C}{\mathbb{C}}
\newcommand{\R}{\mathbb{R}}
\newcommand{\Z}{\mathbb{Z}}

\newcommand{\N}{\mathbb{N}}

\newcommand{\cH}{\mathcal{H}}
\newcommand{\soL}{\mathfrak{so}}
\newcommand{\slL}{\mathfrak{sl}}

\newcommand{\gL}{\mathfrak{g}}
\newcommand{\kL}{\mathfrak{k}}
\newcommand{\mL}{\mathfrak{m}}
\newcommand{\aL}{\mathfrak{a}}
\newcommand{\pL}{\mathfrak{p}}
\newcommand{\qL}{\mathfrak{q}}
\newcommand{\nL}{\mathfrak{n}}
\newcommand{\hL}{\mathfrak{h}}
\newcommand{\bL}{\mathfrak{b}}
\newcommand{\tL}{\mathfrak{t}}

\newcommand{\Real}{\operatorname{Re}}
\newcommand{\Spin}{\operatorname{Spin}}

\newcommand{\SO}{\operatorname{SO}}
\newcommand{\Tr}{\operatorname{Tr}}
\newcommand{\tr}{\operatorname{tr}}
\newcommand{\Id}{\operatorname{Id}}
\newcommand{\Ad}{\operatorname{Ad}}

\newcommand{\vol}{\operatorname{vol}}
\newcommand{\rk}{\operatorname{rank}}
\newcommand{\End}{\operatorname{End}}
\newcommand{\Hom}{\operatorname{Hom}}
\newcommand{\GL}{\operatorname{GL}}
\newcommand{\SL}{\operatorname{SL}}
\newcommand{\Ind}{\operatorname{Ind}}
\newcommand{\bs}{\backslash}

\newcommand{\diag}{\operatorname{diag}}

\newcommand{\Rep}{\operatorname{Rep}}
\renewcommand{\Re}{\operatorname{Re}}

\begin{document}

%Topmatter
\title[]
{Analytic torsion and $L^2$-torsion of compact locally symmetric manifolds}
\date{\today}

\author{Werner M\"uller}
\address{Universit\"at Bonn\\
Mathematisches Institut\\
Endenicher Allee 60\\
D--53115 Bonn, Germany\\}
\email{mueller@math.uni-bonn.de}

\author{Jonathan Pfaff}
\address{Universit\"at Bonn\\
Mathematisches Institut\\
Endenicher Alle 60\\
D--53115 Bonn, Germany\\}
\email{pfaff@math.uni-bonn.de}

\keywords{analytic torsion, locally symmetric manifolds}
\subjclass{Primary: 58J52, Secondary: 53C53}

\begin{abstract}
In this paper we study the analytic torsion and the $L^2$-torsion of compact
locally symmetric manifolds. We consider the analytic torsion with respect to
representations of the fundamental group which are obtained by restriction of
irreducible representations of the group of isometries of the underlying
 symmetric space.
The main purpose is to study the asymptotic behavior of the analytic torsion
with respect to  sequences of representations associated to rays of highest 
weights.
\end{abstract}

\maketitle
\setcounter{tocdepth}{1}
%\tableofcontents
\section{Introduction}
\setcounter{equation}{0}

Let $G$ be a real, connected, linear semisimple Lie group 
with finite center and of  noncompact type. 
Let $K\subset G$ be a maximal compact subgroup. Then $\widetilde X=G/K$ is a 
Riemannian symmetric space of the noncompact type. 
Let $\Gamma\subset G$ be a discrete, torsion free, co-compact subgroup. Then 
$X=\Gamma\bs \widetilde X$ is a compact oriented locally symmetric manifold. 
Let $d=\dim X$. Let $\tau$ be a finite-dimensional irreducible representation
of $G$ on a complex vector space $V_\tau$. Denote by $E_\tau$ the flat vector
bundle over $X$ associated to the representation $\tau|_\Gamma$ of $\Gamma$. 
By \cite[Lemma 3.1]{MtM}, $E_\tau$ can be equipped with a distinguished 
Hermitian
fiber metric, called admissible. Let $\Delta_p(\tau)$ be the Laplace operator
acting on $E_\tau$-valued $p$-forms on $X$. Denote by  
%$\det^\prime\Delta_p(\tau)$ is the zeta function regularized determinant of the
%Laplace operator $\Delta_p(\tau)$. 
%Then the analytic torsion $T_X(\tau)$ is defined by  
%\begin{equation}\label{anator1}
%T_X(\tau)=\prod_{p=1}^d 
%\left(\dett^\prime\Delta_p(\tau)\right)^{(-1)^{p+1}p/2}.
%\end{equation}
$\zeta_p(s;\tau)$ the zeta function of $\Delta_p(\tau)$ (see \cite{Sh}). Then
the analytic torsion $T_X(\tau)\in\R^+$ is defined by
\begin{equation}\label{anator1}
\log T_X(\tau)=\frac{1}{2}\sum_{p=0}^d (-1)^pp\frac{d}{ds}
\zeta_p(s;\tau)\big|_{s=0}
\end{equation}
(see \cite{RS}, \cite{Mu2}). Since we have chosen distinguished metrics, we
don't
indicate the metric dependence of $T_X(\tau)$. We also consider the 
$L^2$-torsion $T_X^{(2)}(\tau)$. Following Lott
\cite{Lo} and Mathai \cite{Mat}, this torsion can be defined using the
$\Gamma$-trace of the heat operators on $\widetilde X$.

The main purpose of this paper is to study the asymptotic behavior of 
$T_X(\tau)$ and $T_X^{(2)}(\tau)$ for certain sequences of representations 
$\tau$ of $G$. This problem was first studied in \cite{Mu3} in the context of 
hyperbolic 3-manifolds.
The method used in this paper was based on the study of the twisted Ruelle 
zeta function. 
In \cite{MP} we have developed a different and simpler method which we
used to extend the results of \cite{Mu3} to compact hyperbolic manifolds 
of any dimension. In the present paper, we generalize the 
results of the previous papers to arbitrary compact locally symmetric spaces. 
Recently, Bismut, Ma, and Zhang \cite{BMZ1}, \cite{BMZ2}, studied the 
asymptotic behavior of 
the analytic torsion by a different method and in the more general context of 
analytic torsion forms on arbitrary compact manifolds. Furthermore, 
Bergeron and Venkatesh \cite{BV} studied the asymptotic behavior of the 
analytic torsion if the flat bundle is kept fixed, but the discrete 
group varies in a tower $\{\Gamma_N\}_{N\in\N}$ of normal subgroups of finite 
index of $\Gamma$. They used this to study the growth of the torsion subgroup
in the cohomology of arithmetic groups. In \cite{MaM} the results of
\cite{Mu3} have been used to study the growth of the torsion in
the cohomology of arithmetic hyperbolic 3-manifolds, if the lattice is kept
fixed and the flat bundle varies. The results of the present paper will be used
to study the growth of the torsion in the cohomology of arithmetic groups in
higher rank cases.  

Now we explain our results in more detail. 
Let $\delta(\widetilde X)=\rk_\C(G)-\rk_\C(K)$. Occasionally we will  
denote this number by $\delta(G)$. Let $\gL$ be the Lie
algebra of $G$. Let
$\hL\subset
\gL$ be a fundamental Cartan subalgebra. Let $G_\C$ denote the connected complex
linear Lie group
corresponding to the complexification $\gL_\C$ of $\gL$ and let 
$U$ be a compact real form of $G_\C$ such that $\hL_\C$ is the complexification 
of a Cartan-subalgebra of $U$. Then the 
irreducible finite dimensional complex representations of $G$ can be
identified with 
the irreducible finite dimensional complex representations of $U$.  
Fix positive roots
$\Delta^+(\gL_\C,\hL_\C)$. Let $\theta\colon\gL\to\gL$ be the Cartan
involution. For a highest weight $\lambda\in\hL^*_\C$ which we always assume 
to be analytically integral with respect to $U$ we let
$\tau_\lambda$ be the irreducible representation 
of $G$ corresponding to the representation of $U$ with highest weight
$\lambda$. 
We will also say that $\tau_\lambda$ is the representation of $G$ of highest
weight $\lambda$. 
Then
we denote by $\lambda_\theta\in\hL^*_\C$ the highest weight of
$\tau_\lambda\circ\theta$,
where 
we regard $\theta$ as an involution on $G$.
Our main result is the following theorem.
\begin{theo}\label{th-main1}
$\mathrm{(i)}$ Let $\widetilde X$ be even dimensional or let $\delta(\widetilde
X)\neq 1$. Then $T_X(\tau)=1$ for all
finite-dimensional representations $\tau$ of $G$.

\noindent
$\mathrm{(ii)} $Let $\widetilde X$ be odd-dimensional with $\delta(\widetilde
X)=1$. Let $\lambda\in\hL^*_\C$ be a 
highest weight with $\lambda_\theta\neq\lambda$. For $m\in\N$ let 
$\tau_\lambda(m)$ be the irreducible representation of $G$ with highest weight 
$m\lambda$. There exist constants $c>0$ and $C_{\widetilde X}\neq 0$, which 
depends on $\widetilde X$, and a polynomial $P_\lambda(m)$, which depends on 
$\lambda$, such that
\begin{equation}\label{asymptor1}
\log T_X(\tau_\lambda(m))=C_{\widetilde X}\vol(X)\cdot P_\lambda(m)
+O\left(e^{-cm}\right)
\end{equation}
as $m\to\infty$. Furthermore, there is a constant $C_\lambda>0$ such that
\begin{equation}\label{polyn1}
P_\lambda(m)=C_\lambda\cdot m\dim(\tau_\lambda(m))+R_\lambda(m),
\end{equation}
where $R_\lambda(m)$ is a polynomial whose degree equals the degree of
the polynomial $\dim(\tau_\lambda(m))$.
\end{theo}
We note that \eqref{asymptor1} provides a complete asymptotic expansion for  
$\log T_X(\tau_\lambda(m))$. If one is only interested in the leading term, one 
can use \eqref{polyn1} which implies that there exists a constant 
$C=C(\widetilde X,\lambda)\neq 0$, 
which depends on $\widetilde X$ and $\lambda$, such that
\begin{equation}\label{asymptor2}
\log T_X(\tau_\lambda(m))=C\vol(X)\cdot m\dim(\tau_\lambda(m))+
O\left(\dim(\tau_\lambda(m))\right)
\end{equation}
as $m\to\infty$. Now the coefficient of the highest power can determined by
Weyl's dimension formula. 

The condition $\lambda\neq\lambda_\theta$ is essential for our method to work.
It implies the existence of an increasing  spectral gap for the corresponding
Laplace operators (see Corollary \ref{lowerbd7}). It is a challenging and
very interesting problem to extend Theorem \ref{th-main1} to the case 
$\lambda=\lambda_\theta$.

For hyperbolic manifolds, we proved the vanishing result (i) of Theorem 
\ref{th-main1} in \cite[Proposition 1.7]{MP}. 
In general it was first proved by Bismut,
Ma, and Zhang \cite{BMZ2}. It extends a  result of Moscovici and Stanton 
\cite{MS1} who showed that $T_X(\rho)=1$, if $\delta(\widetilde X)\ge 2$ and 
$\rho$ is a unitary representation of $\Gamma$. Our proof is different from the
previous proofs and, as we believe, also simpler. It does not rely on the
use of orbital integrals or the Fourier inversion formula. 

Part (ii) is a consequence of
the following two propositions. The first one shows  that  the asymptotic 
behavior of the analytic torsion with respect to the representations 
$\tau_\lambda(m)$ is determined by the asymptotic behavior of the
$L^2$-torsion. 
\begin{prop}\label{prop-l2tor1}
Let $\widetilde X$ be odd-dimensional with $\delta(\widetilde X)=1$. Let
$\lambda\in\hL^*_\C$ be a
highest weight. Assume that $\lambda_\theta\neq\lambda$. For $m\in\N$ let 
$\tau_\lambda(m)$ be the irreducible 
representation of $G$ with highest weight $m\lambda$. Then there exists $c>0$ 
such that
\begin{equation}\label{l2tor-tor}
\log T_X(\tau_\lambda(m))=\log T_X^{(2)}(\tau_\lambda(m))+O\left(e^{-cm}\right)
\end{equation}
for all $m\in\N$.
\end{prop}
This result was first proved in \cite{MP} for hyperbolic manifolds. It was also 
proved in \cite{BMZ2} in the more general context of this paper 
(see Remark 7.8).
Our method of proof of \eqref{l2tor-tor} follows the method developed in 
\cite{MP}.

The key result on which part (ii) of Theorem \ref{th-main1} relies is
the computation of the $L^2$-torsion. The computation is based on the 
Plancherel formula. It gives

\begin{prop}\label{prop-l2tor2}
Let the assumptions be as in Proposition $\mathrm{\ref{prop-l2tor1}}.$
There exists a constant $C_{\widetilde X}$, which depends on
$\widetilde X$, and a polynomial $P_\lambda(m)$, which depends on $\lambda$,
such that
\begin{equation}
\log T^{(2)}_X(\tau_\lambda(m))=C_{\widetilde X}\vol(X)\cdot P_\lambda(m),\quad
m\in\N.
\end{equation}
Moreover there is a constant $C_\lambda>0$ such that 
\begin{equation}\label{polyasym}
P_\lambda(m)=C_\lambda \cdot m\cdot\dim(\tau_\lambda(m))
+O\left(\dim(\tau_\lambda(m)\right)
\end{equation}
as $m\to\infty$.
\end{prop}

Finally, we note that if one specializes the main result of \cite{BMZ2}, 
Theorem 1.1, to the case of analytic torsion of a locally symmetric space, one
can also determine the leading term of the asymptotic expansion of 
\eqref{asymptor2}. This has been carried out in \cite{BMZ2} in the case of
hyperbolic $3$-manifolds.

If we consider one of the odd-dimensional irreducible symmetric spaces
$\widetilde X$ with 
$\delta(\widetilde X)=1$ and choose  $\lambda$ to be (an integral multiple of) a
fundamental weight, 
the statements can be made more explicit. 

Let $G=\SO^0(p,q)$,
$K=\SO(p)\times \SO(q)$, $p>1$, $p,q$ odd, $p\geq q$, and 
let $\widetilde{X}:=G/K$. Let $n:=(p+q-2)/2$. There are two fundamental weights 
$\tilde{\omega}_{f,n}^{\pm}$ which are not invariant under $\theta$ and we let
$\omega_{f,n}^{\pm}:=2\tilde{\omega}_{f,n}^\pm$ (see
\eqref{hiweight}).
One 
has $\omega_{f,n}^{-}=(\omega_{f,n}^+)_\theta$. By equation \eqref{Torstheta},
it suffices to consider the 
weight $\omega_{f,n}^+$. 
For $m\in\N$ let $\tau(m)$ be the representation with highest weight 
$m\omega_{f,n}^+$. By Weyl's dimension formula there exists a constant $C>0$
such that
\begin{equation}
\dim(\tau(m))=Cm^{\frac{n(n+1)}{2}}+O\left(m^{\frac{n(n+1)}{2}-1}\right)
\end{equation}
as $m\to\infty$. Let $\widetilde X_d$ be the compact dual of $\widetilde X$.
We let $\epsilon(q):=0$ for $q=1$ and $\epsilon(q):=1$ for $q>1$ and we let 
\begin{equation}
C_{p,q}:=\frac{(-1)^{\frac{pq-1}{2}}2^{\epsilon(q)}\pi}{\vol(\widetilde{X}
_d)}\begin{pmatrix}n\\
\frac{p-1}{2} \end{pmatrix} .
\end{equation}
\begin{corollary}\label{cor-spin}
Let $p,q$ odd and let $\widetilde X=\SO^0(p,q)/\SO(p)\times \SO(q)$
and let $X=\Gamma\bs \widetilde X$. With respect to the
above notation we have
\[
\log T_X(\tau(m))=C_{p,q}\vol(X)\cdot m\dim(\tau(m))+
O\left(m^{\frac{n(n+1)}{2}}\right)
\]
as $m\to\infty$. 
\end{corollary}
The case $q=1$ was treated in \cite{MP} and the case $p=3$, $q=1$ in \cite{Mu3}.
In the latter case we have $\Spin(3,1)\cong\SL(2,\C)$. The irreducible
representation of $\Spin(3,1)$ with highest weight $\frac{1}{2}(m,m)$ 
corresponds to the $m$-th symmetric power of the standard 
representation $\SL(2,\C)$ on $\C^2$ and we have
\[
-\log T_X(\tau(m))=\frac{1}{4\pi}\vol(X) m^2+O(m).
\]

The remaining case is $\widetilde X=\SL(3,\R)/\SO(3)$. There are two
fundamental weights $\omega_i$, $i=1,2$. Both are non-invariant under $\theta$.
Let $\tau_i(m)$, $i=1,2$, be the irreducible representation with highest 
weight $m\omega_i$. By Weyl's dimensiona formula one has 
\begin{align*}
\dim_{\tau_i}(m)=\frac{1}{2}m^2+O(m),
\end{align*}
as $m\to\infty$. Let $\widetilde X_d$ be the compact dual of $\widetilde X$. 
\begin{corollary}\label{cor-sl3}
Let  $\widetilde X=\SL(3,\R)/\SO(3)$ and $X=\Gamma\bs \widetilde X$. We have
\[
\log T_X(\tau_i(m))=\frac{4\pi\vol(X)}{9\vol(
\widetilde{X}_d)}m\dim(\tau_i(m))+O(m^2)
\]
as $m\to\infty$.
\end{corollary} 
Using the equality of analytic and Reidemeister torsion \cite{Mu2}, we obtain
corresponding statements for the Reidemeister torsion $\tau_X(\tau_\lambda(m))$.
Especially we have
\begin{corollary}
Let $X=\Gamma\bs \widetilde X$ be a compact odd-dimensional locally symmetric
manifold with
$\delta(\widetilde X)=1$. Let $\lambda\in\hL^*_\C$ be a highest weight which
satisfies $\lambda_\theta\neq\lambda$. Let $\tau_X(\tau_\lambda(m))$ be the 
Reidemeister torsion of $X$ with respect to the representation 
$\tau_\lambda(m)$. Then $\vol(X)$ is determined by the set
$\{\tau_X(\tau_\lambda(m))\colon m\in\N\}$.
\end{corollary}
Finally we note that Bergeron and Venkatesh \cite{BV} proved results of a
similar nature, but in a different aspect. Let $\delta(\widetilde X)=1$.
Let $\Gamma\supset\Gamma_1\supset\cdots\supset \Gamma_N\supset\cdots$ be a 
tower of subgroups of finite index with $\cap_N\Gamma_N=\{e\}$. A 
representation $\tau$ of $G$ is called strongly acyclic, if the spectrum
of the Laplacians $\Delta_p(\tau)$ on $\Gamma_N\bs \widetilde X$ stays
uniformly bounded away from zero. Then for a strongly acyclic representation 
$\tau$ they show that there is a constant
$c_{G,\tau}>0$ such that
\[
\lim_{N\to\infty}\frac{\log T_{\Gamma_N\bs \widetilde
X}(\tau)}{[\Gamma\colon\Gamma_N]}
=c_{G,\tau}\vol(\Gamma\bs \widetilde X).
\]

Next we explain our methods to prove Theorem \ref{th-main1}. The first step 
is the proof of Proposition \ref{prop-l2tor1}. We follow the proof
used in \cite{MP}. For an irreducible representation $\tau$ of $G$ and 
$t>0$ put
\[
K(t,\tau):=\sum_{p=0}^d (-1)^p p\Tr\left(e^{-t\Delta_p(\tau)}\right).
\]
Assume that $\tau|_\Gamma$ is acyclic, that is $H^*(X,E_\tau)=0$. Then the 
analytic torsion is given by
\begin{equation}\label{anator}
\log T_X(\tau):=\frac{1}{2}\frac{d}{ds}\left(\frac{1}{\Gamma(s)}
\int_0^\infty t^{s-1}K(t,\tau)\;dt\right)\bigg|_{s=0}.
\end{equation}
Now the key ingredient of the proof of Proposition \ref{prop-l2tor1} is the
following lower bound for the spectrum of the Laplacians.
For every highest weight $\lambda$ which satisfies $\lambda_\theta\neq\lambda$,
there exist $C_1,C_2>0$ such that
\begin{equation}\label{lowerbd0}
\Delta_p(\tau_\lambda(m))\ge C_1m^2-C_2,\quad m\in\N,
\end{equation}
(see Corollary \ref{lowerbd7}). Since $\tau_\lambda(m)$ is acyclic and $\dim X$ 
is odd, $T_X(\tau_\lambda(m))$ is metric independent \cite{Mu2}. Especially, it 
is invariant under rescaling of the metric. So we can replace 
$\Delta_p(\tau_\lambda(m))$ by $\frac{1}{m}\Delta_p(\tau_\lambda(m))$. Then
\begin{equation}\label{anator01}
\begin{split}
\log T_X(\tau(m))=&\frac{1}{2}\frac{d}{ds}\left(\frac{1}{\Gamma(s)}
\int_0^1 t^{s-1}K\left(\frac{t}{m},\tau(m)\right)\,dt\right)\bigg|_{s=0}\\
&+\frac{1}{2}\int_1^\infty t^{-1}K\left(\frac{t}{m},\tau(m)\right)\,dt.
\end{split}
\end{equation}
It follows from \eqref{lowerbd0} and standard estimations of the heat
kernel that the second term on the right is $O(e^{-\frac{m}{8}})$ as
$m\to\infty$.
To deal with the first term, we use a preliminary form of the Selberg trace 
formula. It turns out that the contribution of the nontrivial conjugacy 
classes to the trace
formula is also exponentially decreasing in $m$. Finally, the identity 
contribution equals $\log T^{(2)}_X(\tau_\lambda(m))$ up to a term, which is 
exponentially decreasing in $m$. This implies Proposition 
\ref{prop-l2tor1}.

To deal with the $L^2$-torsion, we recall that for any $\tau$, 
$\log T^{(2)}_X(\tau)$ it is defined in terms of the
$\Gamma$-trace of the heat operators $e^{-t\widetilde \Delta_p(\tau)}$ on the 
universal covering \cite{Lo}, \cite{Mat}. In our case,  $e^{-t\widetilde
\Delta_p(\tau)}$ 
is a 
convolution operator and its $\Gamma$-trace equals the contribution of the 
identity to the spectral side of the Selberg trace formula applied to
$e^{-t\Delta_p(\tau)}$. It follows that
\[
\log T_X^{(2)}(\tau)=\vol(X)\cdot t^{(2)}_{\widetilde X}(\tau),
\]
where $t^{(2)}_{\widetilde X}(\tau)$ depends only on $\widetilde X$ and $\tau$. 
To compute $t^{(2)}_{\widetilde X}(\tau)$ we factorize $\widetilde X$ as
$\widetilde X=\widetilde X_0\times \widetilde X_1$, where 
$\delta(\widetilde X_0)=0$ and $\widetilde X_1$ is irreducible with 
$\delta(\widetilde X_1)=1$. Let $\tau=\tau_0\otimes\tau_1$ be the corresponding
decomposition of $\tau$. Let $\widetilde X_{0,d}$ be the compact dual
symmetric space of $\widetilde X_0$.
Using  a formula similar to \cite[Proposition 11]{Lo}, we get
\[
t_{\widetilde X}^{(2)}(\tau)=(-1)^{\dim(\widetilde X_0)/2}\frac{\chi(\widetilde
X_{0,d})}
{\vol(\widetilde X_{0,d})}\dim(\tau_0)\cdot t_{\widetilde X_1}^{(2)}(\tau_1).
\]
This reduces the computation of 
$t^{(2)}_{\widetilde X}(\tau)$  to the case of an irreducible symmetric
space $\widetilde X$ with $\delta(\widetilde X)=1$ which is odd-dimensional.
From the classification
of simple Lie groups it follows  that the only possibilities for 
$\widetilde X$ are 
$\widetilde X=\SL(3,\R)/\SO(3)$ or $\widetilde X=\SO^0(k,l)/\SO(k)\times\SO(l)$,
$k,l$ odd. Using the Plancherel
formula, $t^{(2)}_{\widetilde
X}(\tau)$
can be computed explicitly for these cases. Combined with Weyl's dimension
formula, it follows that $t^{(2)}_{\widetilde X}(\tau_\lambda(m))$ is a
polynomial in
$m$. In this way we obtain our main result. 

The paper is organized as follows. In section \ref{secprel} we collect some 
facts about representations of reductive Lie groups. Section \ref{secBLO} is
concerned with Bochner-Laplace operators on locally symmetric spaces. The main
result are estimations of the heat kernel of a Bochner-Laplace operator.
In section \ref{sector} we consider the analytic torsion in general. The main 
result of this section is Proposition \ref{vanish1}, which establishes part
(i) of Theorem \ref{th-main1}. Section \ref{secl2-tor} is devoted to the study
of the $L^2$-torsion. We reduce the study of the $L^2$-torsion to the case of 
an irreducible symmetric space $\widetilde X$ with $\delta(\widetilde X)=1$. 
This case is then treated in section \ref{secdelta1}. Especially we establish
Proposition \ref{prop-l2tor2} in this case. In section \ref{seclowbd} we prove
a lower bound for the spectrum of the twisted Laplace operators. This is the
key result for the proof of Proposition \ref{prop-l2tor1}. In the final section 
\ref{secmainres} we prove our main result, Theorem \ref{th-main1}.

\section{Preliminaries}\label{secprel}
\setcounter{equation}{0}

In this section we summarize some facts about representations of reductive
Lie groups. 
\subsection{}

Let $G$ be a real reductive Lie group in the sense of \cite[p. 446]{Kn2}. Let
$K\subset G$ be the associated maximal compact subgroup. Then $G$ has only
finitely many connected components. Denote by $G^0$ the component of the
identity. Let $\gL$ and $\kL$ denote the Lie algebras of $G$ and $K$, 
respectively. Let $\gL=\kL\oplus\pL$ be the Cartan decomposition. 

We denote by $\hat G$ the unitary dual and by $\hat G_d$ the discrete
series of $G$. By $\Rep(G)$ we denote the
equivalence classes of irreducible finite-dimensional representations of $G$. 

Let $Q$ be a standard parabolic subgroup of $G$ \cite[VII.7]{Kn2}.  Then $Q$ has
a
Langlands decomposition $Q=MAN$, where $M$ is reductive and $A$ is abelian. 
$Q$ is called cuspidal if $\hat M_d\neq\emptyset$. Let $K_M=K\cap M$. Then 
$K_M$ is a maximal compact subgroup of $M$. 

Let $Q=MAN$ be cuspidal. For $(\xi,W_\xi)\in\hat M_d$ and $\nu\in\aL^*_\C$,
let
\begin{equation}\label{indrep}
\pi_{\xi,\nu}=\Ind_Q^G(\xi\otimes e^\nu\otimes \Id)
\end{equation}
be the induced representation acting by the left regular representation
on the Hilbert space 
\begin{equation}\label{indrep1}
\begin{split}
\cH_{\xi,\nu}=\bigl\{f\colon &G\to W_\xi\colon f(gman)=e^{-(i\nu+\rho_Q)(\log
a)}
\xi(m)^{-1}f(g),\\
&\forall\,\, m\in M,\,a\in A,\,n\in N,\,g\in G,\,f|_K\in L^2(K,W_\xi)\bigr\}
\end{split}
\end{equation}
with norm given by
\[
\|f\|^2=\int_K|f(k)|^2_{W_\xi}\,dk.
\]
If $\nu\in\aL^*$ , then $\pi_{\xi,\nu}$ is unitarily induced.
 Denote by $\Theta_{\xi,\nu}$ the global character of $\pi_{\xi,\nu}$.

\subsection{}\label{secDS}
Next we recall some facts concerning the discrete series. 
Let $G$ be a linear semisimple connected Lie group with
finite center. Let
$K\subset G$ be a maximal compact subgroup. Assume that $\delta(G)=0$. Then 
$G/K$ is even-dimensional. Let $n=\dim(G/K)/2$. Let $\tL\subset\kL$ be a
compact Cartan subalgebra of $\gL$. Let $\Delta(\gL_{\C},\tL_{\C})$,
$\Delta(\kL_\C,\tL_\C)$ be the corresponding roots with Weyl-groups $W_G$,
$W_K$. 
Then one can regard $W_K$ as a subgroup of $W_G$. Let $P$ be the weight lattice
in $i\tL^*$. Let $\left<\cdot,\cdot\right>$ be the inner product on
$i\tL^*$ induced by the Killing form. Recall that $\Lambda\in P$ is called
regular if 
$\left<\Lambda,\alpha\right>\neq 0$ for all
$\alpha\in\Delta(\gL_{\C},\tL_{\C})$. 
Then $\hat{G}_d$ is parametrized by the $W_K$-orbits of the 
regular elements of $P$, where $W_K$ is the Weyl group of
$\Delta(\kL_\C,\tL_\C)$, \cite[Theorem 12.20, Theorem 9.20]{Kn1}. If $\Lambda$
is a regular
element of 
$P$, the corresponding discrete series will be denoted by $\omega_\Lambda$. 
For $\pi\in\hat G$ we denote by $\chi_\pi$ the infinitesimal character of
$\pi$. For a regular element $\Lambda\in\hL_\C^*$ let $\chi_\Lambda$ be the 
homomorphism of $\mathcal{Z}(\gL_\C)$, defined by \cite[(8.32)]{Kn1}.
By \cite[Theorem 9.20]{Kn1},
the infinitesimal character of $\omega_\Lambda$ is given by $\chi_\Lambda$.
Fix positive roots $\Delta^+(\gL_{\C},\tL_{\C})$ and let
$P^+$ be the 
corresponding set of dominant weights. Let $\rho_G$ be the half sum of the
elements of $\Delta^+(\gL_{\C},\tL_{\C})$
Then we have the following proposition. 
\begin{prop}\label{PropDS}
Let $\tau\in\Rep(G)$. Then for $\pi\in\hat{G}_d$ one has
\begin{align*}
\dim\left(H^p(\gL,K;\cH_{\pi,K}\otimes V_\tau)\right)
=\begin{cases}
1,& \chi_\pi=\chi_{\check{\tau}}, \:p=n;\\
0,& \text{else}.
\end{cases}
\end{align*}
Moreover, there are exactly $|W_G|/|W_K|$ distinct elements of $\hat{G}_d$ with
infinitesimal 
character $\chi_{\check{\tau}}$, where $\check{\tau}$ is the contragredient
representation 
of $\tau$. 
\end{prop}
\begin{proof}
Let $\Lambda(\check{\tau})\in P^+$ be the highest weight of $\tau$. Clearly
$\Lambda(\check{\tau})+\rho_G$ is
regular. Thus, since
$W_G$ acts 
freely on the regular elements, the
proposition follows from 
\cite[Theorem I.5.3]{BW} and the above remarks on infinitesimal characters.
\end{proof}

\subsection{}
Let $Q=MAN$ be a standard parabolic subgroup. In general, $M$ is neither
semisimple nor connected. But $M$ is reductive in the sense of 
\cite[p. 466]{Kn2}. Let $K_M=K\cap M$, let $K_M^0$ be the component of 
the identity, and let $\kL_{\mL}:=\kL\cap\mL$ be its Lie algebra. Assume
that $\rk(M)=\rk(K_M)$. Then $M$ has a nonempty discrete series, which is
defined as in \cite[XII,\S 8]{Kn1}. The explicit parametrization is given in 
\cite[Proposition 12.32]{Kn1}, \cite[section 8.7.1]{Wa1}.

\section{Bochner Laplace operators}\label{secBLO}
\setcounter{equation}{0}

Let $G$ be a semisimple connected Lie group with
finite center. Let $K\subset G$ be a maximal compact subgroup. Let $\widetilde
X=G/K$. Let  $\Gamma$ be a torsion free, cocompact
discrete subgroup of $G$ and let $X=\Gamma\bs \widetilde X$.

Let $\nu$ be a finite-dimensional unitary representation of $K$ on the space
$(V_{\nu},\left<\cdot,\cdot\right>_{\nu})$ let 
\begin{align*}
\widetilde{E}_{\nu}:=G\times_{\nu}V_{\nu}
\end{align*}
be the associated homogeneous vector bundle over $\widetilde{X}$. 
Let $R_g\colon \widetilde{E}_{\nu}\to \widetilde{E}_{\nu}$ be the action of 
$g\in G$. The inner product
$\left<\cdot,\cdot\right>_{\nu}$ induces a $G$-invariant fiber metric 
$\widetilde{h}_{\nu}$ on $\widetilde{E}_{\nu}$. Let $\widetilde{\nabla}^{\nu}$
be the 
connection on $\widetilde{E}_{\nu}$ induced by the canonical connection on the
principal $K$-fiber bundle $G\to G/K$. Then 
$\widetilde{\nabla}^{\nu}$ is $G$-invariant.
Let  
\begin{align*}
E_{\nu}:=\Gamma\bs \widetilde{E}_\nu
\end{align*}
be the associated locally homogeneous bundle over $X$. Since 
$\tilde{h}_{\nu}$ and $\widetilde{\nabla}^{\nu}$ are $G$-invariant, they can be
pushed down to a metric $h_{\nu}$ and a connection $\nabla^{\nu}$ on $E_{\nu}$. 
Let $C^{\infty}(\widetilde{X},\widetilde{E}_{\nu})$ resp. 
$C^{\infty}(X,E_{\nu})$ denote the
space of smooth sections of $\widetilde{E}_{\nu}$
resp. of $E_\nu$.  Let
\begin{align}\label{globsect}
C^{\infty}(G,\nu) 
:=\{&f:G\rightarrow V_{\nu}\colon f\in C^\infty,\nonumber \\ &\;
f(gk)=\nu(k^{-1})f(g),\,\,\forall g\in G, \,\forall k\in K\}.
\end{align}
Let $L^2(G,\nu)$ be the corresponding $L^2$-space.
There is a canonical isomorphism
\begin{equation}\label{iso1}
A:C^{\infty}(\widetilde{X},\widetilde{E}_{\nu})\cong C^{\infty}(G,\nu)
\end{equation}
which is defined by $Af(g)=R_g^{-1}(f(gK))$. It extends to an isometry
\begin{equation}\label{iso2}
A\colon L^{2}(\widetilde{X},\widetilde{E}_{\nu})\cong L^2(G,\nu).
\end{equation}
Let
\begin{align}\label{globsect1}
C^{\infty}(\Gamma\backslash G,\nu):=\left\{f\in C^{\infty}(G,\nu)\colon 
f(\gamma g)=f(g)\:\forall g\in G, \forall \gamma\in\Gamma\right\}
\end{align}
and let $L^2(\Gamma\bs G,\nu)$ be the corresponding $L^2$-space. 
The isomorphisms \eqref{iso1} and \eqref{iso2} descend to isomorphisms
\begin{equation}\label{iso3}
A:C^{\infty}(X,E_{\nu}) \cong C^{\infty}(\Gamma\backslash G,\nu),\quad
L^{2}(X,E_{\nu})\cong L^2(\Gamma\bs G,\nu).
\end{equation}
Let
$\widetilde{\Delta}_{\nu}={\widetilde{\nabla^\nu}}^{*}{\widetilde{\nabla}}^{\nu}
$
be the Bochner-Laplace operator of $\widetilde{E}_{\nu}$. 
Since $\widetilde{X}$ is complete, 
$\widetilde{\Delta}_{\nu}$ with domain the space of smooth compactly supported 
sections
is essentially self-adjoint \cite[p. 155]{LM}. Its self-adjoint extension
will be denoted by $\widetilde{\Delta}_\nu$ too. With respect to the isomorphism
\eqref{iso1} one has
\begin{align}\label{BLO}
\widetilde{\Delta}_{\nu}=-R(\Omega)+\nu(\Omega_K),
\end{align} 
where $R$ denotes the right regular representation 
of $G$ on $C^\infty(G,\nu)$ (see \cite[Proposition 1.1]{Mi1}). The heat
operator 
\[
e^{-t\tilde{\Delta}_{\nu}}\colon L^2(G,\nu)\to L^2(G,\nu)
\]
commutes with the action of $G$. Therefore, it is of the form
\begin{equation}
(e^{-t\tilde{\Delta}_{\nu}}\phi)(g)=\int_G
H_t^\nu(g^{-1}g^\prime)(\phi(g^\prime))\;dg^\prime
\end{equation}
where 
\[
H_t^\nu\colon G\to \End(V_\nu)
\]
is in $C^\infty\cap L^2$ and satisfies the covariance property
\begin{equation}
H^{\nu}_{t}(k^{-1}gk')=\nu(k)^{-1}\circ H^{\nu}_{t}(g)\circ\nu(k'),
\:\forall k,k'\in K, \forall g\in G.
\end{equation}
It follows as in \cite[Proposition 2.4]{Barbasch}  that
 $H^\nu_t$ belongs to all Harish-Chandra Schwartz spaces
$(\mathcal{C}^{q}(G)\otimes {\rm{End}}(V_{\nu}))$, $q>0$.

Now let $\left\|H_t^\nu(g)\right\|$ be the sup-norm of $H_t^\nu(g)$ in
$\End(V_\nu)$. Let $\widetilde{\Delta}_0$ be the Laplacian on functions on 
$\widetilde{X}$ and let $H^0_t$ be  the associated heat kernel as above.
We may use the principle of semigroup domination to bound
$\left\|H_t^{\nu}(g)\right\|$ by the scalar heat kernel. Indeed we have
\begin{prop}\label{Kato}
Let $\nu\in\hat{K}$. Then we have
\[
\left\|H_t^{\nu}(g)\right\|\leq H^0_t(g)
\]
for all $t\in\R^+$ and $g\in G$. 
\end{prop}
\begin{proof}
Let $K_\nu(t,x,y)$ be the kernel of $e^{-t\widetilde{\Delta}_\nu}$, acting in 
$L^2(\widetilde{X},\widetilde{E}_\nu)$. Denote by
$|K_\nu(t,x,y)|$ the norm of the homomorphism 
\[
K_\nu(t,x,y)\in\Hom\left((\widetilde{E}_\nu)_y,(\widetilde{E}_\nu)_x\right).
\]
It was proved in \cite[p. 325]{Mu1} that in the sense of distributions, one has
\[
\left(\frac{\partial}{\partial t}+\widetilde\Delta_{0}\right)
|K_\nu(t,x,y)|\le 0,
\]
where $\widetilde\Delta_0$ acts in the $x$-variable. Using (3.15) in \cite{Mu1}
one can proceed as in the proof of Theorem 4.3 of \cite{DL} to show that
\begin{equation}\label{kato1}
|K_\nu(t,x,y)|\le K_0(t,x,y),\quad t\in\R^+,\,x,y\in\widetilde X,
\end{equation}
where $K_0(t,x,y)$ is the kernel of $e^{-t\widetilde{\Delta}_0}$. See also 
\cite[p. 7]{Gu}. Now observe that
\[
H_t^\nu(g^{-1}g^\prime)=R_g^{-1}\circ K_\nu(t,gK,g^\prime K)\circ
R_{g^\prime}\:\textup{and}\:
H_t^0(g^{-1}g^\prime)=K_0(t,gK,g^\prime K).
\]
Since for each $x\in\widetilde{X}$, 
$R_g\colon (\widetilde{E}_{\nu})_x\to (\widetilde{E}_{\nu})_{g(x)}$ is an
isometry,
the proposition follows from \eqref{kato1}.
\end{proof}

Now we pass to the quotient $X=\Gamma\backslash\widetilde{X}$.
Let $\Delta_\nu={\nabla^\nu}^*\nabla^\nu$ be the Bochner-Laplace
operator. It is essentially self-adjoint. Let $R_\Gamma$ be the right regular 
representation of $G$ on $C^\infty(\Gamma\bs G,\nu)$. By \eqref{BLO} it follows 
that with respect to the isomorphism \eqref{iso3} we have
\begin{equation}
\Delta_\nu=-R_\Gamma(\Omega)+\nu(\Omega_K).
\end{equation}
Let $e^{-t\Delta\nu}$ be the heat semigroup of $\Delta_\nu$, acting on
$L^2(\Gamma\backslash G,\nu)$. Then $e^{-t\Delta_\nu}$ is represented by the 
smooth kernel
\begin{align}\label{KernX}
H_\nu(t,g,g^\prime):=\sum_{\gamma\in\Gamma}H_t^\nu(g^{-1}\gamma g^\prime).
\end{align}
The convergence of the series in \eqref{KernX} can be
established, for example, using Proposition \ref{Kato} and the methods from 
the proof of Proposition \ref{esthyp} below. Put
\begin{align}\label{hnu}
h_t^\nu(g):=\tr H_t^\nu(g),\quad g\in G,
\end{align}
where $\tr\colon \End(V_\nu)\to\C$ is the matrix trace.
Then the trace of the heat operator $e^{-t\Delta_\nu}$ is given by
\begin{align}\label{TrBLO}
\Tr(e^{-t\Delta_\nu})=\int_{\Gamma\bs G} \tr H_\nu(t,g,g)\;d\dot
g=\int_{\Gamma\bs G}
\sum_{\gamma\in\Gamma}h_t^\nu(g^{-1}\gamma g)d\dot g.
\end{align}
Using results of Donnelly we now prove an estimate for the heat kernel $H_t^0$
of the Laplacian $\widetilde\Delta_0$ acting on $C^\infty(\widetilde{X})$. 
\begin{prop}\label{esthyp}
There exist constants $C_0$ and
$c_0$
such that for every $t\in\left(0,1\right]$ and every $g\in G$ one has
\begin{align*}
\sum_{\substack{\gamma\in\Gamma\\ \gamma\neq
1}}H_t^0(g^{-1}\gamma g)\leq C_0 e^{-c_0/t}.
\end{align*}
\end{prop}
\begin{proof}
For $x,y\in\widetilde{X}$ let $\rho(x,y)$ denote the geodesic distance of $x,y$.
Since $K(t,gK,g^\prime K)=H^0_t(g^{-1}g^\prime)$ is the kernel of 
$e^{-t\widetilde{\Delta}_0}$, it follows from \cite[Theorem 3.3]{Do1} that there
exists a constant $C_1$ such that
for every $g\in G$ and every $t\in\left(0,1\right]$ one has
\begin{align}\label{Don}
H_t^0(g)\leq C_1 t^{-\frac{d}{2}}\exp{\left(-\frac{\rho^2(gK,1K)}{4t}\right)}.
\end{align}
Let $x\in\widetilde{X}$ and let $B_R(x)$ be the metric ball around $x$ of radius
$R$. Let $h>0$ be the topological entropy of the geodesic flow of $X$ (see
\cite{Ma}). There exists $C_2>0$ such that 
\begin{align}\label{volgr}
\vol B_R(x)\leq C_2 e^{hR},\quad R>0
\end{align}
\cite{Ma}. Since $\Gamma$ is cocompact and torsion-free, there exists an 
$\epsilon>0$ such that
$B_\epsilon(x)\cap \gamma B_\epsilon(x)=\emptyset$ for every 
$\gamma\in\Gamma-\{1\}$ and every $x\in\widetilde{X}$. Thus for every 
$x\in \widetilde{X}$ the union over
all
$\gamma B_\epsilon (x)$,
where $\gamma\in\Gamma$ is such that $\rho(x,\gamma x)\leq R$ 
is disjoint and is contained in $B_{R+\epsilon}(x)$.
Using \eqref{volgr} it follows that there
exists a constant $C_3$ such that for every $x\in\widetilde{X}$ one has
\begin{align*}
\#\{\gamma\in\Gamma\colon \rho(x,\gamma x)\leq R\}\leq C_3 e^{hR}. 
\end{align*}
Hence there exists a
constant $C_4>0$ such that for every $x\in \widetilde{X}$ one has
\begin{align}\label{sum}
\sum_{\substack{\gamma\in\Gamma\\ \gamma\neq 1}}e^{-\frac{\rho^{2}(\gamma
x,x)}{8}}\leq
C_4.
\end{align}
Now let
\begin{align*}
c_1:=\inf\{\rho(x,\gamma x)\colon \gamma\in\Gamma-\{1\},\:x\in\widetilde{X}\}.
\end{align*}
We have $c_1>0$. Using \eqref{Don} and \eqref{sum}, it follows that there are
constants $c_0>0$ and $C_0>0$ such that for every $g\in G$ and $0< t\le 1$  we 
have
\begin{align*}
\sum_{\substack{\gamma\in\Gamma\\ \gamma\neq
1}}H_t^0(g^{-1}\gamma g)\leq
C_1t^{-\frac{d}{2}}e^{-c_1^2/(8t)}\sum_{\substack{\gamma\in\Gamma\\ \gamma\neq
1}}
e^{-\rho^{2}(\gamma gK,gK)/8}\leq C_0e^{-c_0/t}.
\end{align*}
\end{proof}

\section{The analytic torsion}\label{sector}
\setcounter{equation}{0}

Let $\tau$ be an irreducible finite-dimensional representation of $G$ on $
V_{\tau}$. Let $E_{\tau}$ be the flat vector bundle over $X$ associated to the 
restriction of $\tau$ to $\Gamma$. Let $\widetilde E^\tau$ be the homogeneous
vector bundle associated to $\tau|_K$ and let $E^\tau:=\Gamma\bs \widetilde 
E^\tau$. There is a canonical isomorphism
\begin{equation}\label{vectbdl}
E^\tau\cong E_\tau
\end{equation}
\cite[Proposition 3.1]{MtM}. By \cite[Lemma 3.1]{MtM}, there exists an 
inner product $\left<\cdot,\cdot\right>$ on $V_{\tau}$ such that 
\begin{enumerate}
\item $\left<\tau(Y)u,v\right>=-\left<u,\tau(Y)v\right>$ for all 
$Y\in\mathfrak{k}$, $u,v\in V_{\tau}$
\item $\left<\tau(Y)u,v\right>=\left<u,\tau(Y)v\right>$ for all 
$Y\in\mathfrak{p}$, $u,v\in V_{\tau}$.
\end{enumerate}
Such an inner product is called admissible. It is unique up to scaling. Fix an 
admissible inner product. Since $\tau|_{K}$ is unitary with respect to this 
inner product, it induces a metric on $E^{\tau}$, and by \eqref{vectbdl} 
on $E_\tau$, which we also call 
admissible. Let $\Lambda^{p}(E_{\tau})=\Lambda^pT^*(X)\otimes E_\tau$. Let
\begin{align}\label{repr4}
\nu_{p}(\tau):=\Lambda^{p}\Ad^{*}\otimes\tau:\:K\rightarrow{\rm{GL}}
(\Lambda^{p}\mathfrak{p}^{*}\otimes V_{\tau}).
\end{align}
Then there is a canonical isomorphism
\begin{align}\label{pforms}
\Lambda^{p}(E_{\tau})\cong\Gamma\backslash(G\times_{\nu_{p}(\tau)}
\Lambda^{p}\mathfrak{p}^{*}\otimes V_{\tau}).
\end{align}
of locally homogeneous vector bundles. 
Let  $\Lambda^{p}(X,E_{\tau})$ be the space the smooth $E_{\tau}$-valued 
$p$-forms on $X$. The isomorphism \eqref{pforms} induces an isomorphism
\begin{align}\label{isoschnitte}
\Lambda^{p}(X,E_{\tau})\cong C^{\infty}(\Gamma\backslash G,\nu_{p}(\tau)),
\end{align}
where the latter space is defined as in \eqref{globsect1}. A corresponding 
isomorphism also holds for the spaces of $L^{2}$-sections. 
Let $\Delta_{p}(\tau)$ be the 
Hodge-Laplacian on $\Lambda^{p}(X,E_{\tau})$ with respect to the admissible 
metric in $E_\tau$. Let $R_\Gamma$ denote the right regular representation 
of $G$ in $L^2(\Gamma\bs G)$. By \cite[(6.9)]{MtM} it follows that with
respect to the isomorphism \eqref{isoschnitte} one has
\[
\Delta_{p}(\tau)f=-R_\Gamma(\Omega)f+\tau(\Omega)\Id f, \: f\in
C^{\infty}(\Gamma\backslash G,\nu_{p}(\tau)).
\]
We remark that in \cite{MtM} it is not assumed that $G$ does not have compact
factors, see the remark on page 372 of \cite{MtM},  
and so we do not make this assumption either. Let
\begin{align}\label{heattor1}
K(t,\tau):=\sum_{p=1}^{d}(-1)^{p}p\Tr(e^{-t\Delta_{p}(\tau)}).
\end{align}
and
\begin{equation}\label{constanttau}
h(\tau):=\sum_{p=1}^{d}(-1)^p p \dim H^p(X,E_\tau).
\end{equation}
Then $K(t,\tau)-h(\tau)$ decays exponentially as $t\to\infty$ and it follows 
from  \eqref{anator1} that
\begin{align}\label{anator2}
\log{T_{X}(\tau)}=\frac{1}{2}\frac{d}{ds}\left(\frac{1}{\Gamma(s)}
\int_{0}^{\infty}t^{s-1}(K(t,\tau)-h(\tau))\;dt\right)\bigg|_{s=0},
\end{align}
where the right hand side is defined near $s=0$ by analytic continuation of 
the Mellin transform. Let
$\widetilde{E}_{\nu_p(\tau)}:=G\times_{\nu_p(\tau)}\Lambda^p\mathfrak{p}
^*\otimes V_{\tau}$ and let
$\widetilde{\Delta}_p(\tau)$ be the lift of $\Delta_p(\tau)$ to
$C^\infty(\widetilde{X},\widetilde{E}_{\nu_p(\tau)})$.
Then again it follows from \cite[(6.9)]{MtM} that on $C^\infty(G,\nu_p(\tau))$
one has
\begin{align}\label{kuga}
\tilde\Delta_p(\tau)=-R_\Gamma(\Omega)+\tau(\Omega)\Id.
\end{align}
Let $e^{-t\tilde\Delta_p(\tau)}$
be the corresponding heat semigroup on $L^2(G,\nu_p(\tau))$. 
It is
a smoothing operator which commutes with the action of $G$. Therefore, it is of
the form
\begin{displaymath}
 \left( e^{-t\tilde\Delta_p(\tau)}\phi\right)(g)=\int_G  
H^{\tau,p}_t(g^{-1}g')\phi(g') \;dg',\quad 
\phi\in(L^2(G,\nu_p(\tau)),\;\;g\in G,
\end{displaymath}
where the kernel
\begin{align}\label{DefHH}
H^{\tau,p}_t\colon G\to\End(\Lambda^p\mathfrak p^*\otimes
V_\tau)
\end{align} 
belongs to $C^\infty\cap L^2$ and  satisfies the covariance property
\begin{equation}\label{covar}
H^{\tau,p}_t(k^{-1}gk')=\nu_p(\tau)(k)^{-1} H^{\tau,p}_t(g)\nu_p(\tau)(k')
\end{equation}
with respect to the representation \eqref{repr4}.
Moreover, for all $q>0$ we have 
\begin{equation}\label{schwartz1}
H^{\tau,p}_t \in (\mathcal{C}^q(G)\otimes
\End(\Lambda^p\pL^*\otimes V_\tau))^{K\times K}, 
\end{equation}
where $\mathcal{C}^q(G)$ denotes Harish-Chandra's $L^q$-Schwartz space. 
The proof is similar to the proof of Proposition 2.4 in \cite{Barbasch}. 
Now we come to the heat kernel of $\Delta_p(\tau)$. First 
the integral kernel of $e^{-t\Delta_p(\tau)}$, regarded as an operator in 
$L^2(\Gamma\backslash G,\nu_p(\tau))$, is given by 
\begin{align}\label{kernelx}
H^{\tau,p}(t;g,g'):=\sum_{\gamma\in\Gamma}{H}^{\tau,p}_{t}(g^{-1}\gamma g'),
\end{align}
As in section \ref{secBLO} this series converges absolutely and locally 
uniformly. Therefore the trace of the heat operator $e^{-t\Delta_p(\tau)}$ is 
given by
\[
\Tr\left(e^{-t\Delta_p(\tau)}\right)=\int_{\Gamma\bs G}\tr
H^{\tau,p}(t;g,g)\;d\dot g,
\]
where $\tr$ denotes the trace $\tr\colon \End(V_\nu)\to \C$. Let
\begin{align}\label{Defh}
{h}^{\tau,p}_{t}(g):=\tr{H}^{\tau,p}_{t}(g).
\end{align}
Using \eqref{kernelx} we obtain
\begin{equation}\label{TrDeltap}
\Tr\left(e^{-t\Delta_p(\tau)}\right)=\int_{\Gamma\bs G}\sum_{\gamma\in\Gamma} 
{h}^{\tau,p}_{t}(g^{-1}\gamma g)\,d\dot g.
\end{equation}
Put
\begin{equation}\label{alter}
  k_t^\tau=\sum_{p=1}^d(-1)^pp\, h^{\tau,p}_t.
\end{equation}
Then it follows that
\begin{equation}\label{anator3}
K(t,\tau)=\int_{\Gamma\bs G}\sum_{\gamma\in\Gamma} 
k_{t}^\tau(g^{-1}\gamma g)\,d\dot g.
\end{equation}
Let $R_\Gamma$ be the right regular representation of $G$ on $L^2(\Gamma\bs G)$.
Then \eqref{anator3} can be written as
\begin{equation}\label{trace8}
K(t,\tau)=\Tr R_\Gamma(k^\tau_t).
\end{equation}
We shall now compute the Fourier transform of $k^\tau_t$.
To begin with let $\pi$ be an admissible unitary 
representation of $G$ on a Hilbert space $\cH_\pi$. Set
\[
\tilde \pi(H^{\tau,p}_t)=\int_G \pi(g)\otimes H^{\tau,p}_t(g)\,dg.
\]
This defines a bounded operator on 
$\cH_\pi\otimes \Lambda^p\mathfrak p^*\otimes V_\tau$. As in 
\cite[pp. 160-161]{Barbasch} it follows from 
\eqref{covar} that relative to the splitting
\[
\cH_\pi\otimes \Lambda^p\mathfrak p^*\otimes V_\tau=
\left(\cH_\pi\otimes \Lambda^p\mathfrak p^*\otimes V_\tau\right)^K\oplus
\left[\left(\cH_\pi\otimes \Lambda^p\mathfrak p^*\otimes V_\tau\right)^K
\right]^\perp,
\]
$\tilde \pi(H^{\tau,p}_t)$ has the form
\[
\tilde \pi(H^{\tau,p}_t)=\begin{pmatrix}\pi(H^{\tau,p}_t)& 0\\ 0& 0
\end{pmatrix}
\]
with $\pi(H^{\tau,p}_t)$ acting on $\left(\cH_\pi\otimes 
\Lambda^p\mathfrak p^*\otimes V_\tau\right)^K$. Using \eqref{kuga} it follows 
as in \cite[Corollary 2.2]{Barbasch} that
\begin{equation}\label{equtrace2}
\pi(H^{\tau,p}_t)=e^{t(\pi(\Omega)-\tau(\Omega))}\Id
\end{equation}
on $\left(\cH_\pi\otimes \Lambda^p\mathfrak p^*\otimes V_\tau\right)^K$.
Let $\{\xi_n\}_{n\in\N}$ and $\{e_j\}_{j=1}^m$ be orthonormal bases of $\cH_\pi$
and  $\Lambda^p\mathfrak p^*\otimes V_\tau$, respectively. Then we have
\begin{equation}\label{equtrace1}
\begin{split}
\Tr\pi(H_t^{\tau,p})&=\sum_{n=1}^\infty\sum_{j=1}^m\langle\pi(H_t^{\tau,p})
(\xi_n\otimes e_j),(\xi_n\otimes e_j)\rangle\\
&=\sum_{n=1}^\infty\sum_{j=1}^m\int_G\langle\pi(g)\xi_n,\xi_n\rangle
\langle H_t^{\tau,p}(g)e_j,e_j\rangle\,dg\\
&=\sum_{n=1}^\infty\int_G h_t^{\tau,p}(g)\langle\pi(g)\xi_n,\xi_n\rangle\;dg\\
&=\Tr\pi(h_t^{\tau,p}).
\end{split}
\end{equation}
Let $\pi\in\hat G$ and let $\Theta_\pi$ denote its character.
Then it follows
from \eqref{alter}, \eqref{equtrace2} and \eqref{equtrace1} that
\begin{equation}\label{equtrace3}
\Theta_\pi(k_t^\tau)=e^{t(\pi(\Omega)-\tau(\Omega))}\sum_{p=1}^d (-1)^p p\cdot
\dim(\cH_\pi\otimes\Lambda^p\pL^*\otimes V_\tau)^K.
\end{equation}
Now we consider the case of a principle series representation. Let $Q$ be a
standard
cuspidal parabolic subgroup.  Let $Q=MAN$ be the Langlands
decomposition of $Q$. Denote by $\aL$ the Lie algebra of $A$. Let $K_M=K\cap M$.
Let $(\xi,W_\xi)$ be a discrete series representation of 
$M$ and let $\nu\in \aL^*_{\C}$. Let $\pi_{\xi,\nu}$ be the induced 
representation and let $\Theta_{\xi,\nu}$ be the character of $\pi_{\xi,\nu}$
(see
section \ref{secprel}). 
\begin{prop}\label{equcharac}
Let $Y\in\aL$ be a unit vector and let $\pL_Y$ be the orthogonal complement
of $Y$ in $\pL$. Then
\begin{enumerate}
\item[(i)] $\Theta_{\xi,\nu}(k_t^\tau)=e^{t(\pi_{\xi,\nu}(\Omega)-\tau(\Omega))}
\dim\left(W_\xi\otimes(\Lambda^{\mathrm{odd}}\pL_Y^*-\Lambda^{\mathrm{ev}}
\pL_Y^*)\otimes 
V_\tau\right)^{K_M}$,
\item[(ii)] $\Theta_{\xi,\nu}(k_t^\tau)=0$ if $\dim\aL_\qL\ge 2$.
\end{enumerate}
\end{prop}
\begin{proof} 
By Frobenius reciprocity \cite[p. 208]{Kn1} and \eqref{equtrace3} we get
\[
\Theta_{\xi,\nu}(k_t^\tau)=e^{t(\pi_{\xi,\nu}(\Omega)-\tau(\Omega))}\sum_{p=1}^d
(-1)^p p
\dim\left(W_\xi\otimes\Lambda^p\pL^*\otimes V_\tau\right)^{K_M}.
\]
Now 
\[
\pL^*=\R Y^*\oplus\pL^*_Y
\]
as $K_M$-module. Therefore, in the Grothendieck ring of $K_M$ 
we have
\begin{equation}\label{equp}
\begin{split}
\sum_{p=1}^d(-1)^p p \Lambda^p\pL^*&=\sum_{p=1}^d(-1)^p p\Lambda^p\pL_Y^*\oplus
\Lambda^{p-1}\pL_Y^*\\
&=\sum_{p=1}^d(-1)^p
p\Lambda^p\pL_Y^*+\sum_{p=0}^{d-1}(-1)^{p+1}(p+1)\Lambda^p\pL_Y^*\\
&=\sum_{p=0}^d(-1)^{p+1}\Lambda^p\pL_Y^*.
\end{split}
\end{equation}
Tensoring with $W_\xi$ and $V_\tau$ and taking $K_M$-invariants, we 
obtain (i). 

To prove (ii), suppose that there is a nonzero $H\in\aL\cap \pL_Y$. Since
$M$ centralizes $H$, $\varepsilon(H)+i(H)$ is a $K_M$ intertwining
operator between $\Lambda^{\mathrm{ev}}\pL_Y^*$ and
$\Lambda^{\mathrm{odd}}\pL_Y^*$,
and 
non-trivial since $H\neq0$. Hence $\Lambda^{\mathrm{ev}}\pL_Y^*$ and 
$\Lambda^{\mathrm{odd}}\pL_Y^*$ are equivalent as $K_M$-modules and (ii)
follows.
\end{proof}

\begin{prop}\label{vanish1}
Assume that $\delta(\tilde{X})\geq 2$ or assume that $\tilde{X}$ is
even-dimensional. 
Then $T_X(\tau)=1$ 
for all finite-dimensional irreducible representations $\tau$ of $G$.
\end{prop}
\begin{proof}
Let 
\[
R_\Gamma=\bigoplus_{\pi\in\hat G} m_\Gamma(\pi)\pi
\]
be the decomposition of the right regular representation $R_\Gamma$ of $G$ on
$L^2(\Gamma\backslash G)$, see \cite[section 1]{Wa0}. Now insert $k_t^\pi$ 
on both sides and apply the trace. If we recall that
$\Theta_\pi(k_t^\tau)=\tr\pi(k_t^\tau)$ and use \eqref{trace8}, we get
\begin{align}\label{serK}
K(t,\tau)=\sum_{\pi\in\hat G} m_\Gamma(\pi)\Theta_\pi(k_t^\tau).
\end{align}
The series on the right hand side is absolutely convergent. 
First assume that $\delta(X)\geq 2$. By \cite[section 2.2]{Del} the Grothendieck
group of all 
admissible representations of $G$ is generated by the representations
$\pi_{\xi,\lambda}$ , where 
$\pi_{\xi,\lambda}$ is associated to some standard cuspidal parabolic subgroup
$Q$ of $G$ as in \eqref{indrep}. 
Since $\delta(X)\geq
2$ one has 
$\Theta_{\xi,\lambda}(k_t^\tau)=0$ for every such representation  by Proposition
\ref{equcharac}. Thus one has $\Theta_\pi(k_t^\tau)=0$ 
for every irreducible unitary representation of $G$. By \eqref{serK} it follows
that
$K(t,\tau)=0$. Let $h(\tau)$ be as in \eqref{constanttau}. Since 
$K(t,\tau)-h(\tau)$ decays exponentially as $t\to\infty$, 
it follows that $K(t,\tau)-h(\tau)=0$ and using \eqref{anator2}, the first
statement follows. 

Now assume that $d=\dim \widetilde X$ is even. Note that as $K$-modules we have
\[
\Lambda^p\pL^*\cong\Lambda^{d-p}\pL^*,\quad p=0,\dots,d.
\]
Since $d$ is even, it follows that in the representation ring $R(K)$ we have 
the following equality
\[
\sum_{p=0}^d (-1)^pp\Lambda^p\pL^*=\frac{d}{2}\sum_{p=0}^d(-1)^p\Lambda^p\pL^*.
\]
Let $(\pi,\cH_\pi)\in\hat G$. Then it follows from 
\eqref{equtrace3} that
\[
\Theta_\pi(k_t^\tau)=\frac{d}{2}e^{t(\pi(\Omega)-\tau(\Omega)}\sum_{p=0}^d
(-1)^p
\dim(\cH_\pi\otimes\Lambda^p\pL^*\otimes V_\tau)^K.
\]
Let $\cH_{\pi,K}$ be the subspace of $\cH_\pi$ consisting of all smooth
$K$-finite
vectors. Then
\[
(\cH_{\pi,K}\otimes\Lambda^p\pL^*\otimes V_\tau)^K=(\cH_\pi\otimes\Lambda^p\pL^*
\otimes V_\tau)^K.
\]
Thus the $(\gL,K)$-cohomology $H^*(\gL,K;\cH_{\pi,K}\otimes V_\tau)$ is 
computed from the Lie algebra cohomology complex $([\cH_\pi\otimes\Lambda^p\pL^*
\otimes V_\tau]^K,d)$ (see \cite{BW}). Using the Poincar\'e principle we get
\begin{equation}\label{charcoh}
\Theta_\pi(k_t^\tau)=\frac{d}{2}e^{t(\pi(\Omega)-\tau(\Omega)}\sum_{p=0}^d
(-1)^p
\dim H^p(\gL,K;\cH_{\pi,K}\otimes V_\tau).
\end{equation}
Now by \cite[II.3.1, I.5.3]{BW} we have 
\begin{equation}\label{eqcoh}
H^p(\gL,K;\cH_{\pi,K}\otimes V_\tau)=\begin{cases}
[\cH_\pi\otimes\Lambda^p\pL^*\otimes V_\tau]^K,& \pi(\Omega)=\tau(\Omega);\\
0,& \pi(\Omega)\neq \tau(\Omega).
\end{cases}
\end{equation}
Hence for every $\pi\in \hat G$ one has 
$\Theta_\pi(k_t^\tau)\in \Z$ and $\Theta_\pi(k_t^\tau)$ is independent of $t>0$.
Since the series on the right hand side of \eqref{serK} converges absolutely, 
there exist only finitely many $\pi\in\hat G$ 
with $m_\Gamma(\pi)\neq 0$ and $\Theta_\pi(k_t^\tau)\neq 0$. Thus
$K(t,\tau)$ is independent of $t>0$. Let $h(\tau)$ be defined by 
\eqref{constanttau}. Then $K(t,\tau)-h(\tau)=O(e^{-ct})$ as $t\to\infty$. Hence
$K(t,\tau)=h(\tau)$. By \eqref{anator2} it follows that $T_X(\tau)=1$.
\end{proof}

\section{$L^2$-torsion}\label{secl2-tor}
\setcounter{equation}{0}

In this section we study the $L^2$-torsion $T^{(2)}_X(\tau)$. For its definition
we refer to \cite{Lo}, \cite{Mat}. Actually, in \cite{Lo} and \cite{Mat} only
the case of the trivial
representation $\tau_0$ has been discussed. However the extension to a 
nontrivial $\tau$ is straight forward. The definition is based on the 
$\Gamma$-trace of the heat operator $e^{-t\widetilde\Delta_p(\tau)}$ on the
universal covering $\widetilde X$ (see \cite{Lo}, \cite{Mat}). For our purposes,
it suffices
to introduce the $L^2$-torsion
for representations $\tau$ on $\widetilde{X}$ which satisfy 
$\tau_\theta\not\cong\tau$. \\
Let
$h_t^{\tau,p}$ be the
function defined by 
\eqref{Defh}. By homogeneity it follows that in our case the 
$\Gamma$-trace is given by
\begin{equation}\label{gammatr}
\Tr_\Gamma\left(e^{-t\widetilde\Delta_p(\tau)}\right)=\vol(X) h_t^{\tau,p}(1).
\end{equation} 
In order to define the $L^2$-torsion we need to know the asymptotic behavior
of $h_t^{\tau,p}(1)$ as $t\to 0$ and $t\to\infty$. First we consider the 
behavior as $t\to 0$. Using \eqref{TrDeltap} we have
\begin{equation}\label{trace19}
\vol(X)h_t^{\tau,p}(1)=\Tr\left(e^{-t\Delta_p(\tau)}\right)-
\int_{\Gamma\bs G}\sum_{\gamma\in\Gamma-\{1\}}h_t^{\tau,p}(g^{-1}\gamma g)\,
d\dot g.
\end{equation}
To deal with the second term on the right, we consider the representation
$\nu_p(\tau)$  of $K$ which is defined by \eqref{repr4}, and for $p=0,\dots,n$ 
we put
\begin{equation}\label{endom2}
E_p(\tau):=\tau(\Omega)\Id-\nu_p(\tau)(\Omega_K),
\end{equation}
which we regard as endomorphism of $\Lambda^p\pL^*\otimes V_\tau$. It defines an
endomorphism of $\Lambda^pT^*(X)\otimes E_\tau$. By \eqref{BLO} and 
\eqref{kuga} we have 
\begin{equation}\label{bochhodge1}
\Delta_p(\tau)=\Delta_{\nu_p(\tau)}+E_p(\tau).
\end{equation}
Let $\nu_p(\tau)=\oplus_{\sigma\in\hat K}m(\sigma)\sigma$ be the decomposition
of 
$\nu_p(\tau)$ into irreducible representations. This induces a corresponding
decomposition of the homogeneous vector bundle
\[
\widetilde E_{\nu_p(\tau)}=\bigoplus_{\sigma\in\hat K}m(\sigma)\widetilde
E_\sigma.
\]
With respect to this decomposition we have
\begin{equation}\label{bhdecomp1}
E_p(\tau)=\bigoplus_{\sigma\in\hat K}
m(\sigma)\left(\tau(\Omega)-\sigma(\Omega_K)\right)\Id_{V_\sigma},
\end{equation}
where $\sigma(\Omega_K)$ is the Casimir eigenvalue of $\sigma$ and $V_\sigma$ is
the representation space of $\sigma$, and
\begin{equation}\label{bhdecomp2}
\Delta_{\nu_p(\tau)}=\bigoplus_{\sigma\in\hat K} m(\sigma)\Delta_\sigma.
\end{equation}
This shows that $\Delta_{\nu_p(\tau)}(\tau)$ commutes with $E_p(\tau)$. Let
$H_t^{\nu_p(\tau)}$ be the kernel of $e^{-t\widetilde\Delta_{\nu_p(\tau)}}$ and
let 
$H_t^{\tau,p}$ be the kernel of $e^{-t\widetilde\Delta_p(\tau)}$. Using
\eqref{bochhodge1} we get
\begin{equation}\label{heatfact}
H_t^{\tau,p}(g)=e^{-tE_p(\tau)}\circ H_t^{\nu_p(\tau)}(g),\quad g\in G.
\end{equation}
Let $c\in\R$ be such that $E_p(\tau)\ge c$. 
By Proposition \ref{Kato} it follows that
\begin{equation}
\| H_t^{\tau,p}(g)\|\le e^{-ct} H_t^0(g),\quad t\in\R^+,\,g\in G.
\end{equation}
Taking the trace in $\End(\Lambda^p\pL^*\otimes V_\tau)$ we get
\begin{equation}
\begin{split}
&\sum_{\gamma\in\Gamma-\{1\}}|h^{\tau,p}_t(g^{-1}\gamma g)| \\ &\le
\begin{pmatrix}d\\
p\end{pmatrix}\dim(\tau)
e^{-ct} 
\sum_{\gamma\in\Gamma-\{1\}} H_t^0(g^{-1}\gamma g),\quad t\in\R^+,\,g\in G.
\end{split}
\end{equation}
By Proposition \ref{esthyp} there exist $C_1,c_1>0$ such that 
\begin{equation}
\int_{\Gamma\bs G}\sum_{\gamma\in\Gamma-\{1\}}|h^{\tau,p}_t(g^{-1}\gamma
g)|\,d\dot g
\le C_1 e^{-c_1/t}
\end{equation}
for $0< t\le 1$. Thus by \eqref{trace19}
\[
h_t^{\tau,p}(1)=\frac{1}{\vol(X)}\Tr\left(e^{-t\Delta_p(\tau)}\right)+O(e^{
-c_1/t})
\]
for $0< t\le 1$. Using the asymptotic expansion of $\Tr\left(
e^{-t\Delta_p(\tau)}\right)$ (see \cite{Gi}), it follows that there is an
asymptotic
 expansion
\begin{equation}\label{asympexp3}
h_t^{\tau,p}(1)\sim \sum_{j=0}^\infty a_j t^{-d/2+j}
\end{equation}
as $t\to 0$. To study the behavior of $h_t^{\tau,p}(1)$ as $t\to\infty$, we use 
the Plancherel theorem, which can be applied since $h_t^{\tau,p}$ is a
$K$-finite Schwarz function. Let $\pi$ be an admissible unitary representation
of 
$G$ on a Hilbert space $\cH_\pi$. It follows from \eqref{equtrace2} and 
\eqref{equtrace1} that
\[
\Tr\pi(h_t^{\tau,p})=e^{t(\pi(\Omega)-\tau(\Omega))}\dim\left(\cH_\pi\otimes
\Lambda^p\pL^*\otimes V_\tau\right)^K. 
\]
Let $Q=MAN$ be a standard parabolic subgroup of $G$. Let $(\xi,W_\xi)$ be a
discrete series representation of $M$. Let $\left<\cdot,\right>$ denote the
inner product 
on the real vector space $\aL^*$ induced by the Killing form. Fix positive
restricted roots 
of $\aL$ and let $\rho_\aL$ denote the corresponding half-sum of these roots.
Define a constant
$c(\xi)$ by
\begin{align}\label{defcon}
c(\xi):=-\langle\rho_\aL,\rho_\aL\rangle+
\xi(\Omega_M).
\end{align}
Recall that for $\nu\in \aL^*$ one has
\begin{align}\label{InfchHS}
\pi_{\xi,\nu}(\Omega)=-\langle\nu,\nu\rangle+c(\xi).
\end{align}
Then by the Plancherel theorem,
\cite[Theorem 3]{HC} and
\eqref{InfchHS} we 
have
\begin{align*}
&h_t^{\tau,p}(1)\\ =&\sum_Q\sum_{\xi\in\hat M_d}e^{-t(\tau(\Omega)-c(\xi))}
\int_{\aL^*}e^{-t\|\nu\|^2}
\dim\left(\cH_{\xi,i\nu}\otimes\Lambda^p\pL^*\otimes V_\tau\right)^K
p_\xi(i\nu)\;d\nu.
\end{align*}
Here the outer sum is over all association classes of standard cuspidal
parabolic subgroups of $G$ and $p_\xi(i\nu)$, the Plancherel-density, is of
polynomial growth in $\nu$.
Let $K_M=K\cap M$. By Frobenius reciprocity we have 
\begin{align}\label{FrRec}
\dim\left(\cH_{\xi,\nu}\otimes\Lambda^p\pL^*\otimes V_\tau\right)^K=
\dim\left(W_\xi\otimes\Lambda^p\pL^*\otimes V_\tau\right)^{K_M}.
\end{align}
Thus we get
\begin{equation}\label{planch12}
\begin{split}
&h_t^{\tau,p}(1)\\ =&\sum_Q\sum_{\xi\in\hat M_d}
\dim\left(W_\xi\otimes\Lambda^p\pL^*\otimes V_\tau\right)^{K_M}
e^{-t(\tau(\Omega)-c(\xi))}\int_{\aL^*}e^{-t\|\nu\|^2}p_\xi(i\nu)\;d\nu.
\end{split}
\end{equation}
The exponents of the exponential factors in front of the integrals are
controlled by the following lemma.

\begin{lem}\label{exponents}
Let $(\tau,V_\tau)\in\Rep(G)$.
Assume that $\tau\not\cong\tau_\theta$. 
Let $P=MAN$ be a cuspidal parabolic subgroup of $G$. Let $\xi\in\hat{M}_d$ and
assume that $\dim\left(W_\xi\otimes\Lambda^p\pL^*\otimes V_\tau\right)^{K_M}\neq
0$. Then one has
\begin{align*}
\tau(\Omega)-c(\xi)>0.
\end{align*}
\end{lem}
\begin{proof}
Assume that $\tau(\Omega)-c(\xi)\leq 0$. Then by \eqref{InfchHS} there exists a
$\nu_0\in\aL^*$
such that
\begin{align*}
\pi_{\xi,\nu_0}(\Omega)=\tau(\Omega).
\end{align*}
Together with \eqref{FrRec}, our assumption and \cite[Proposition II.3.1]{BW} it
follows that
\begin{align*}
\begin{split}
\dim\left(H^p(\gL,K;\cH_{\xi,\nu_0,K}\otimes V_\tau)\right)\neq 0,
\end{split}
\end{align*}
where $\cH_{\xi,\nu_0,K}$ are the $K$-finite vectors in $\cH_{\xi,\nu_0}$
Since $\tau\not\cong\tau_\theta$, this is a contradiction to the first 
statement of \cite[Proposition II. 6.12]{BW}. 
\end{proof}

Let $\tau$ be an irreducible representation of $G$ which satisfies 
$\tau\not\cong\tau_\theta$. It follows from \eqref{planch12} and
Lemma \ref{exponents} that there exists $c>0$ such that
\begin{equation}\label{toinft}
h_t^{\tau,p}(1)=O\left(e^{-ct}\right)
\end{equation}
as $t\to\infty$. Using \eqref{asympexp3} and \eqref{toinft} it follows from
standard methods, see for example \cite{Gi}, that the 
Mellin transform
\[
\int_0^\infty h_t^{\tau,p}(1)t^{s-1}\;dt
\]
converges absolutely and uniformly on compact subsets of the half-plane
$\Re(s)>d/2$ and admits a meromorphic extension to $\C$ which is holomorphic
at $s=0$ if $d=\dim(\tilde{X})$ is odd and has  at most a simple pole at $s=0$
for $d=\dim(\tilde{X})$ even. Thus we can define the $L^2$-torsion
$T_X^{(2)}(\tau)\in\R^+$ by  
\begin{equation}\label{l2tor1}
\begin{split}
&\log T_X^{(2)}(\tau)\\ &:=\frac{1}{2}\sum_{p=1}^d(-1)^p p\frac{d}{ds}\left(
\frac{1}{\Gamma(s)}\int_0^\infty
\Tr_\Gamma\left(e^{-t\widetilde\Delta_p(\tau)}\right)
t^{s-1}\,dt\right)\bigg|_{s=0},
\end{split}
\end{equation}
where the right hand side is defined near $s=0$ by analytic continuation. 
For $t>0$ let 
\begin{equation}\label{l2torker}
K^{(2)}(t,\tau):=\sum_{p=1}^d (-1)^pp h_t^{\tau,p}(1).
\end{equation}
Put 
\begin{equation}\label{l2tor3}
t^{(2)}_{\widetilde X}(\tau):=\frac{1}{2}\frac{d}{ds}\left(\frac{1}{\Gamma(s)}
\int_0^\infty K^{(2)}(t,\tau) t^{s-1}\,dt\right)\bigg|_{s=0}.
\end{equation}
Then  $t^{(2)}_{\widetilde X}(\tau)$ depends only on the symmetric space 
$\widetilde X$ and $\tau$, and we have
\begin{equation}\label{l2tor2}
\log T_X^{(2)}(\tau)=\vol(X)\cdot t^{(2)}_{\widetilde X}(\tau).
\end{equation}

Next we establish an auxiliary result concerning the twisted Euler
characteristic. We let $\tau\in\Rep(G)$ be arbitrary. Let
$\cH^p(X,E_\tau):=\ker\Delta_p(\tau)$ be the space of
$E_\tau$-valued harmonic $p$-forms. Let
\[
\chi(X,E_\tau):=\sum_{p=0}^d (-1)^p \dim\cH^p(X,E_\tau)
\]
be the twisted Euler characteristic. Furthermore, let $\tilde{X}_d$ denote 
the compact dual of $\tilde{X}$. The following proposition is a familar
consequence of the index theorem. For the convenience of the reader we include
an independent proof.
\begin{prop}\label{propeqEu}
If $\delta(\tilde{X})\neq 0$, we have $\chi(X,E_\tau)=0$. If 
$\delta(\tilde{X})=0$,  one has
\begin{align}\label{eqEu2}
\chi(X,E_\tau)=(-1)^n\vol(X)\frac{\chi(\tilde{X}_d)}{\vol(\tilde{X}_d)}
\dim(\tau),
\end{align}
where $n=\dim(\widetilde X)/2$. 
\end{prop}
\begin{proof}
Let  $\pi\in\hat{G}$. It follows from \eqref{equtrace2} and
\eqref{equtrace3} that
\[
\sum_{p=0}^d(-1)^p\Theta_{\pi}(h_t^{p,\tau})=e^{
t(\pi(\Omega)-\tau(\Omega))}
\sum_{p=0}^{d}(-1)^p\dim(\cH_\pi\otimes\Lambda^p\pL^*\otimes V_\tau)^K.
\]
Using \cite[II.3.1]{BW}  and the Poincar\'e principle as in the proof of 
Proposition \ref{vanish1}, we get
\begin{equation}\label{FTalt}
\sum_{p=0}^d(-1)^p\Theta_{\pi}(h_t^{p,\tau})=
\sum_{p=0}^{d}(-1)^p\dim H^p(\gL,K;\cH_{\pi,K}\otimes V_\tau).
\end{equation}
Now by \cite[Theorem I.5.3]{BW} it follows that 
if $H^p(\gL,K;\cH_{\pi,K}\otimes V_\tau)\neq
0$, then $\chi_\pi=\chi_{\check{\tau}}$, where $\check{\tau}$ is
the contragredient representation of $\tau$. By \cite[Corollary 10.37, Corollary
9.2]{Kn1} there are only finitely many representations
$\pi\in\hat{G}$ with a given infinitesimal character. Thus if $Q=MAN$ is a
fundamental
parabolic subgroup with $Q\neq G$ and if $\xi\in\hat{M}_d$, it follows that
there are only finitely many $\lambda\in\aL^*$ such that
\begin{align}\label{altsum}
\sum_{p=0}^d(-1)^p\Theta_{\xi,\lambda}(h_t^{p,\tau})\neq 0.
\end{align}
Hence by the Plancherel-Theorem, \cite[Theorem 3]{HC} and \eqref{FTalt} we get
\begin{align}\label{auxeq}
\sum_{p=0}^d(-1)^p h_t^{p,\tau}(1)=\sum_{p=0}^d(-1)^p\sum_{\pi\in\hat{G}_d}
d(\pi)\dim H^p(\gL,K;\cH_{\pi,K}\otimes V_\tau),
\end{align}
where $\hat{G}_d$ denotes the discrete series of $G$ and $d(\pi)$ denotes the
formal degree of $\pi$. The sum is finite. 
Let
\[
b_p^{(2)}(X,E_\tau):=\lim_{t\to\infty}\Tr_\Gamma\left(e^{-t\tilde\Delta_p(\tau)}
\right)
\]
be the $L^2$-Betti number. Using that  \eqref{auxeq} is independent of $t$ 
and \eqref{gammatr}, we get
\begin{align}\label{auxeq1}
\vol(X)\sum_{p=0}^d(-1)^ph_t^{p,\tau}(1)=\sum_{p=0}^d (-1)^p
b^{(2)}_p(X,E_\tau)=
\chi^{(2)}(X,E_\tau).
\end{align}
By the $\Gamma$-index theorem of Atiyah \cite{At} we have $\chi^{(2)}(X,E_\tau)=
\chi(X,E_\tau)$. Hence by \eqref{auxeq} and \eqref{auxeq1} we get
\begin{equation}
\chi(X,E_\tau)=\vol(X)\cdot\sum_{p=0}^d(-1)^p\sum_{\pi\in\hat{G}_d}
d(\pi)\dim H^p(\gL,K;\cH_{\pi,K}\otimes V_\tau).
\end{equation}
If $\delta(\tilde{X})\neq 0$ then $\hat{G}_d$ is empty and hence, this 
sum equals zero, which proves the first statement. Now assume that 
$\delta(\tilde{X})=0$. Then $\tilde{X}$ is 
even-dimensional.
Let $\dim(\tilde{X})=2n$. We keep the notation from section
\ref{secDS}. 
By \cite[Corollary 5.2]{Ol} for $\Lambda'=w(\Lambda(\check{\tau})+\rho_G)$,
$w\in W_G/W_K$ one has
\begin{align*}
d(\omega_{\Lambda'})=\frac{\dim(\tau)}{\vol(\tilde{X}_d)}
\end{align*}
and so together with Proposition \ref{PropDS} we get
\begin{equation}\label{eulerch}
\begin{split}
&\sum_{p=0}^d(-1)^p\sum_{\pi\in\hat{G}_d}d(\pi)\dim H^p(\gL,K;\cH_{\pi,K}\otimes
V_\tau)\\
=&(-1)^{n}\frac{1}{\vol(\tilde{X}_d)}\#(W_G/W_K)\dim(\tau).
\end{split}
\end{equation}
Finally, by \cite[page 175]{Bott} one has
\begin{align*}
\#(W_G/W_K)=\chi(\tilde{X}_d).
\end{align*}
Applying equation \eqref{eulerch}, the proof of the Proposition follows.
\end{proof}

\begin{bmrk}
We remark that if $X$ is Hermitian and $\tau$ is the trivial representation,
then 
equation \eqref{eqEu2} reduces to Hirzebruch's Proportionality principle.
\end{bmrk}

Now we assume that $\delta(\widetilde X)=1$ and that $\widetilde X$ is
odd-dimensional. By the classification of simple Lie
groups we have $\widetilde X=\widetilde X_0\times \widetilde X_1$, where
$\delta(\widetilde X_0)=0$ and $\widetilde X_1=\SL(3,\R)/\SO(3)$ or
$\widetilde X_1=\SO^0(p,q)/\SO(p)\times\SO(q)$, $p,q$ odd. 
Let $\widetilde X_0=
G_0/K_0$ and let $G_1=\SL(3,\R)$, $K_1=\SO(3)$ or 
$G_1=\SO^0(p,q)$, $K_1=\SO(p)\times\SO(q)$, $p,q$ odd.
 Let $G=G_0\times G_1$.
Let $\tau$ be a finite-dimensional irreducible representation of $G$ and 
assume that $\tau\not\cong\tau_\theta$. Then $\tau=\tau_0\otimes\tau_1$, where
$\tau_i$ is an irreducible representation of $G_i$, $i=0,1$, and 
$\tau_1\not\cong\tau_{1,\theta}$.

\begin{prop}\label{l2torprod}
Let $\delta(\widetilde X)=1$ and let $\widetilde{X}$ be odd-dimensional.
Let $\widetilde X=\widetilde X_0\times 
\widetilde X_1$, where $\widetilde X_1$ is an odd-dimensional irreducible
symmetric space 
with $\delta(\widetilde X_1)=1$.
Let $\tau$ be a finite-dimensional irreducible
representation of $G$ with $\tau\not\cong\tau_\theta$. Then
\[
t_{\widetilde X}^{(2)}(\tau)=(-1)^{\dim\widetilde X_0/2}\frac{\chi(\widetilde
X_{0,d})}
{\vol(\widetilde X_{0,d})}\dim\tau_0\cdot t_{\widetilde X_1}^{(2)}(\tau_1).
\]
\end{prop} 
\begin{proof}
Let $\widetilde E\to \widetilde X$ be the homogeneous vector bundle associated 
to $\tau|_K$. Similarly, let $\widetilde E_i\to \widetilde X_i$ be the 
homogeneous vector bundle associated to $\tau_i|_{K_i}$, $i=0,1$. Then
$\widetilde E\cong \widetilde E_1\boxtimes \widetilde E_2$ and
\[
\Lambda^k(\widetilde X,\widetilde E)\cong\bigoplus_{p+q=k}\left(\Lambda^p(
\widetilde X_0,\widetilde E_0)\otimes \Lambda^q(\widetilde X_1,\widetilde E_1)
\right).
\]
With respect to this decomposition we have
\[
\widetilde \Delta_k(\tau)=\bigoplus_{p+q=k}\left(\widetilde \Delta_p(\tau_0)
\otimes\Id +\Id\otimes \widetilde \Delta_q(\tau_1)\right).
\]
Let $H_t^{\tau,k}$ and $H_t^{\tau_i,p}$, $i=0,1$, be the corresponding heat
kernels.
Then it follows that $H_t^{\tau,k}=\oplus_{p+q=k}H_t^{\tau_0,p}\otimes
H_t^{\tau_1,q}$.
Hence for $h_t^{\tau,k}=\tr H_t^{\tau,k}$ and $h_t^{\tau_i,p}=\tr
H_t^{\tau_i,p}$, 
$i=0,1$, we have 
\[
h_t^{\tau,k}=\sum_{p+q=k} h_t^{\tau_0,p}\cdot h_t^{\tau_1,q}.
\]
Using this equality, we get
\begin{equation}\label{dirprod}
\begin{split}
\sum_{k=0}^d(-1)^k k
\,h_t^{\tau,k}(1)&=\sum_{p=0}^{d_1}\sum_{q=0}^{d_2}(-1)^{p+q}(p+q)
\,h_t^{\tau_1,p}(1)\cdot h_t^{\tau_2,q}(1)\\
&=\sum_{p=0}^{d_1}(-1)^p\, h_t^{\tau_1,p}(1)\cdot \sum_{q=0}^{d_2}(-1)^q q 
\,h_t^{\tau_2,q}(1)\\
&+\sum_{q=0}^{d_2}(-1)^q\, h_t^{\tau_2,q}(1)\cdot \sum_{p=0}^{d_1}(-1)^p p
\,h_t^{\tau_1,p}(1).
\end{split}
\end{equation}
Let $\Gamma_i\subset G_i$, $i=0,1$, any cocompact, torsion free discrete
subgroup. The existence of the $\Gamma_i$ follows from our assumptions 
on the $G_i$ stated in the introduction and from results of Borel \cite{Borel}.
Put $X_i=\Gamma_i\bs\widetilde X_i$ and $E_i=\Gamma\bs \widetilde
E_i$.
By \eqref{auxeq1} and the remark following it we have
\begin{equation}\label{l2euler}
\sum_{p=0}^d (-1)^p h_t^{\tau_i,p}(1)=\frac{\chi(X_i)}{\vol(X_i)},\quad i=0,1.
\end{equation}
Taking the Mellin transform of \eqref{dirprod} and using \eqref{l2euler}
and Proposition \ref{propeqEu}, the proposition follows.
\end{proof}

\section{The asymptotics of the $L^2$-torsion for
$\delta(\widetilde X)=1$}\label{secdelta1}
\setcounter{equation}{0}

In this section we study the asymptotic behaviour of the $L^2$-torsion of an
odd-dimensional irreducible symmetric 
space $\widetilde X$ with $\delta(\widetilde X)=1$.
Then we can assume that $G=\SO^0(p,q)$, $p,q$ odd, and
$K=\SO(p)\times\SO(q)$, or $G=\SL_3(\R)$ and
$K=\SO(3)$. To compute the
$L^2$ torsion in these cases, we need some preparation.
Let $P=MAN$ be a fundamental parabolic subgroup of $G$, i.e. we have
$\dim(A)=1$. Let $M^0$ be the identity component of $M$ and let $\mL$ be its Lie
algebra.
Then 
in our case $\mL$ is always semisimple. Let $K_{M}:=
K\cap M$, let $K_M^0$
be 
the identity component of $K_{M}$ and let
$\kL_{\mL}:=\kL\cap\mL$ 
be its Lie algebra.  Let $\tL$ be a Cartan subalgebra of $\kL_{\mL}$. Then $\tL$
is also a Cartan subalgebra of $\mL$ and of $\kL$. Moreover 
$\hL:=\aL\oplus\tL$ is a Cartan subalgebra of $\gL$.

Let $\Delta(\gL_{\C},\hL_{\C})$, $\Delta(\mL_{\C},\tL_{\C})$,
$\Delta((\kL_{\mL})_{\C},\tL_{\C})$ be 
the corresponding roots. Then
there 
is a canonical inclusion 
$\Delta(\mL_{\C},\tL_{\C}
)\hookrightarrow\Delta(\gL_{\C},\hL_{\C})$.
Fix a positive restricted root $e_1\in\aL^*$ and fix positive roots
$\Delta^+(\mL_{\C},\tL_{\C})$.
In this way we obtain positive roots
$\Delta^+(\gL_{\C},\hL_{\C})$. Let
$\rho_{G}$ resp. $\rho_{M}$ be the
half sums of the
elements of $\Delta^+(\gL_{\C},\hL_{\C})$ and $\Delta^+(\mL_{\C},\tL_{\C})$,
respectively. By our choices we have 
$\rho_{G}|_{\mL}=\rho_{M}$.

Let 
\begin{align*}
T:=\{m\in K_M\colon \Ad(m)|_{\tL}=\Id\}.
\end{align*}
Then we have 
\begin{align*}
T=\{k\in K\colon\Ad(k)|_{\tL}=\Id\}.
\end{align*}
Thus $T$ is connected. Let 
$N_{K_M}$ and $N_{K^0_M}$  be the normalizers of $\tL$ in $K_M$ and $K_M^0$,
respectively. 
Let $W_{K_M}:=N_{K_M}/T$ and let $W_{\kL_{\mL}}=N_{K_M^0}/T$ be the Weyl group
of 
$\Delta((\kL_{\mL})_{\C},\tL_{\C})$. Moreover we let
$W_{\mL}$ be the Weyl group of $\Delta(\mL_\C,\tL_\C)$.
Finally we let 
\begin{align*}
W(A):=\{k\in K\colon \Ad(k)\aL=\aL\}/K_M. 
\end{align*}
The following lemma is certainly well-known and has already been used by 
Olbrich, \cite[page 15]{Ol}. However, for the sake of completeness, we include a
proof here.
\begin{lem}\label{LemWA}
One has
\begin{align*}
\frac{|W_{K_M}|}{|W_{\kL_{\mL}}|}\cdot|W(A)|=2.
\end{align*}
\end{lem}
\begin{proof}
By \cite[Proposition 7.19 (b)]{Kn2}, the order 
$\#\left(M/M^0\right)$ equals $\#\left(K_M/K_{M}^0\right)$.
Let $k\in K_M$. Then $\Ad(k)\tL$ is a maximal torus in $\kL_{\mL}$ and thus
there 
exists a $k^0\in K_M^0$ such that $\Ad(k)\tL=\Ad(k^0)\tL$. Hence every element
of 
$K_M/K_M^0$ has a representative in $N_{K_M}$ and thus there are canonical
isomorphisms
$K_M/K_M^0\cong N_{K_M}/N_{K_M^0}\cong W_{K_M}/W_{\kL_{\mL}}$. In other words
$|W_{K_M}|/|W_{\kL_{\mL}}|$ equals the number of components of $M$.
Let $\aL_{\pL}$ be a maximal abelian subspace of $\pL$ containing $\aL$, let
$\Delta_{\aL_{\pL}}$ 
be the  corresponding restricted roots and
let $W(\Delta_{\aL_{\pL}})$ be 
the corresponding Weyl-group. One has
$W(\Delta_{\aL_{\pL}})=N_K(\aL_{\pL})/Z_K(\aL_{\pL})$, where $N_K(\aL_{\pL})$ 
resp. $Z_K(\aL_{\pL})$ are the normalizer resp. centralizer of $\aL_\pL$ in $K$.
Moreover by 
\cite[Proposition 8.85]{Kn2} each element of $W(A)$ has a representative in
$N_K(\aL_{\pL})$, i.e 
can be extended to an element of $W(\Delta_{\aL_{\pL}})$ which fixes $\aL$. Now
a case-by-case study easily implies 
that $W(\Delta_{\aL_{\pL}})$ contains such an element which is non-trivial if
and only if $G=\SO^0(p,1)$. In this case $M$ is connected. In
all other cases,
$M$ has exactly two components.
This proves the Lemma. 
\end{proof}
Let $H_1\in\aL$ with $e_1(H_1)=1$. Then we normalize the Killing form $B$ by
$1/B(H_1,H_1)$. We let $\left\|\cdot\right\|$ be
the corresponding norm on the real vector-space $i\tL^*\oplus\aL^*$.
Let $\Omega$ be the Casimir element with 
respect to the normalized Killing form. Then for $\tau\in\Rep(G)$ with highest
weight 
$\Lambda(\tau)$ we have
\begin{align}\label{casimir1}
\tau(\Omega)=\left\|\Lambda(\tau)+\rho_G\right\|^2-\left\|\rho_G\right\|^2. 
\end{align}
The restriction of $\left<\cdot,\cdot\right>$ to $\mL$ is non-degenerate and
$\Ad$-invariant. Let $\Omega_{M}$ be the corresponding Casimir element.
For $\sigma\in\Rep(M^0)$ with highest weight 
$\Lambda(\sigma)\in i\tL^*$ we define 
\begin{align}\label{def-c}
c(\sigma):=\left\|\Lambda(\sigma)+\rho_M\right\|^2-\left\|\rho_G\right\|^2.
\end{align}
Then one has $c(\sigma)=\chi_\sigma(\Omega_M)-\left\|\rho_G|_{\aL}\right\|^2$
and 
thus one has 
\begin{align}\label{cdual}
c(\sigma)=c(\check{\sigma})
\end{align}
for every $\sigma\in\Rep(M^0)$.  
Let $W_{\gL}:=W(\gL_\C,\hL_\C)$ be the Weyl group of
$\Delta(\gL_\C,\hL_\C)$  
and for $w\in W_{\gL}$ let $\ell(w)$ be its length with respect to the simple
roots
defined by $\Delta^+(\gL_{\C},\hL_{\C})$.
Finally let
\begin{align*}
W^1:=\{w\in W_{\gL}\colon w^{-1}\alpha>0\quad \forall \alpha \in
\Delta^+(\mL_{\C},\tL_{\C})\}.
\end{align*}
The subspace $\nL$ is even-dimensional and we write  $\dim(\nL)=2n$. For
$k=0,\dots,2n$ let $H^k(\nL;V_\tau)$ be the Lie-algebra cohomology
of $\nL$ with 
coefficients in $V_\tau$. Then the $H^k(\nL;V_\tau)$ are $M^0A$-modules and
their decomposition into 
irreducible $M^0A$-components can be described by the following theorem of
Kostant. 
\begin{prop}\label{Prop Kostant}
In the sense of $M^0A$-modules one has
\begin{align*}
H^{k}(\mathfrak{n};V_{\tau})\cong\sum_{\substack{w\in W^1\\ 
\ell(w)=k}}V_{\tau(w)},
\end{align*}
where $V_{\tau(w)}$ is the $M^0A$ module with highest weight $w(\Lambda(\tau)
+\rho_{G})-\rho_{G}$.
\end{prop}
\begin{proof}
See for example \cite[Theorem 2.5.1.3]{Wr}.
\end{proof}

\begin{corollary}\label{Kostant}
As $M^0A$-modules we have
\begin{align*}
\bigoplus_{k=0}^{2n}(-1)^{k}\Lambda^{k}\nL^*
\otimes V_\tau=\bigoplus_{w\in W^1}(-1)^{l(w)}V_{\tau(w)}.
\end{align*}
\end{corollary}
\begin{proof}
This follows from Proposition \ref{Prop Kostant} and the Poincar\'e 
principle \cite[(7.2.3)]{Kostant}.
\end{proof}

For $w\in W^{1}$ let $\sigma_{\tau,w}\in\Rep(M^0)$ be the finite-dimensional
irreducible representation of $M^0$ with highest weight 
\begin{equation}\label{sigmatauw}
\Lambda(\sigma_{\tau,w}):=w(\Lambda(\tau)
+\rho_{G})|_{\mathfrak{t}}-\rho_{M},
\end{equation}
and let  $\lambda_{\tau,w}\in\R$ be such that 
\begin{equation}\label{lambdatauw}
w(\Lambda(\tau)+\rho_{G})|_{\mathfrak{a}}=\lambda_{\tau,w}e_{1}.
\end{equation}
Then we have the following corollary about the Casimir eigenvalue.
\begin{prop}\label{propCE}
For every $w\in W^1$ one has
\begin{align*}
\tau(\Omega)=\lambda_{\tau,w}^2+c(\sigma_{\tau,w}).
\end{align*}
\end{prop}
\begin{proof}
By \eqref{casimir1} we have
\begin{align*}
\tau(\Omega)=&\left\|\Lambda(\tau)+\rho_G\right\|^2-\left\|\rho_G\right\|^2
=\left\|w(\Lambda(\tau)+\rho_G)\right\|^2-\left\|\rho_G\right\|^2\\
=&\left\|\lambda_{\tau,w}e_1\right\|^2+\left\|\Lambda(\sigma_{\tau,w}
)+\rho_M\right\|^2
-\left\|\rho_G\right\|^2=\lambda_{\tau,w}^2+c(\sigma_{\tau,w}).
\end{align*}
\end{proof}
Let $k_t^\tau$ be defined by \eqref{alter}. Our next goal is to
compute the Fourier transform of $k_t^\tau$. Note 
that, since $T$ is connected, 
it follows from \cite[section 6.9, section 8.7.1]{Wa1} that for
every discrete series representation $\xi$ of $M$ over $W_\xi$ there exists  a 
discrete
series representation $\xi^0$ of $M^0$ over $W_{\xi^0}$ such 
that $\xi$ is induced from $\xi^0$. 
Moreover, since $M^0$ is semisimple, the discrete series of $M^0$ is
parametrized as in section \ref{secDS}. By \cite[section 8.7.1]{Wa1}, two 
discrete 
series representations $\xi_1^0$ and $\xi_2^0$ of $M^0$ with corresponding
parameters $\Lambda_{\xi_1^0}, \Lambda_{\xi_2^0}$ as in section \ref{secDS}
induce 
the same discrete series representation 
of $M$ if and only if $\Lambda_{\xi_1^0}$ and $\Lambda_{\xi_2^0}$ are
$W_{K_M}$-conjugate. For $\xi\in\hat{M}_d$ and $\lambda\in\C$ we let 
$\pi_{\xi,\lambda}:=\pi_{\xi,\lambda e_1}$,
$\Theta_{\xi,\lambda}:=\Theta_{\xi,\lambda e_1}$.
\begin{prop}\label{propthetaxi}
Let $\xi\in\hat{M}_d$ with infinitesimal character $\chi(\xi)$. Let
$\pL_{\mL}:=\pL\cap\mL$ and let
$v:=\frac{1}{2}\dim\pL_{\mL}$.
Then for $\lambda\in\C$ one has
\begin{align*}
\Theta_{\xi,\lambda}(k_t^\tau)=
(-1)^v\sum_{\substack{w\in W^1\\\chi(\xi)=\chi(\check{\sigma}_{\tau,w})}}
(-1)^{\ell(w)+1}e^{-t(\lambda^2+\lambda_{\tau,w}^2)}
\end{align*}
\end{prop}
\begin{proof}
Let $\xi^0$, $\Lambda_{\xi^0}$ be as above. Then one has
\begin{align*}
\pi_{\xi,\lambda}
(\Omega)=-\lambda^2+\left\|\Lambda_\xi\right\|^2-\left\|\rho_G\right\|^2.
\end{align*}
Thus if $\sigma\in\Rep(M^0)$ is such 
that $\chi_\sigma=\chi_\xi$ one has
\begin{align}\label{infchtheta}
\pi_{\xi,\lambda}
(\Omega)=-\lambda^2+c(\sigma).
\end{align} 
Moreover
$\xi|_{K_M}$ is induced 
from $\xi^0|_{K^0_{M}} $ and so by Frobenius reciprocity one has 
\begin{align*}
\left[\Lambda^p\pL^*\otimes\cH_\xi\otimes
V_\tau\right]^{K}=\left[\Lambda^p\pL^*\otimes W_\xi\otimes
V_\tau\right]^{K_M}=\left[\Lambda^p\pL^*\otimes W_{\xi_0}\otimes
V_\tau\right]^{K^0_M}.
\end{align*}
Thus by \eqref{equtrace2} one
has
\begin{align*}
\Theta_{\xi,\lambda}(k_t^\tau)=e^{t(\pi_{\xi,\lambda}(\Omega)-\tau(\Omega))}
\sum_{p=0}
^d(-1)^pp\left[\Lambda^p\pL^*\otimes W_{\xi^0}\otimes V_\tau\right]^{K^0_M}.
\end{align*}
Let $\pL_Y$ be as in Proposition \ref{equcharac}. Since $\dim\aL=1$, it follows 
that as $K_M^0$ modules $\pL_Y\cong\pL_{\mL}\oplus\nL$. Using \eqref{equp}, it
follows that as $K_M^0$ modules we have
\begin{align*}
&\sum_{p=0}^d(-1)^pp\Lambda^p\pL^*=\sum_{p=0}^{d}(-1)^{p+1}\Lambda^p(\pL_{\mL}
^*\oplus\nL^*)\\ =&\sum_{k=0}^{2n}(-1)^{k+1}\left(\Lambda^{\mathrm{ev}}\pL_{\mL}
^*-\Lambda^{\mathrm{odd}}\pL_{\mL}
^*\right)\otimes\Lambda^k\nL^*.
\end{align*}
Thus together with Corollary \ref{Kostant} and the Poincar\'e principle one gets
\begin{align*}
&\sum_{p=0}
^d(-1)^pp\left[\Lambda^p\pL^*\otimes W_{\xi^0}\otimes V_\tau\right]^{K^0_M}\\
=
&\sum_{w\in W^1}(-1)^{\ell(w)+1}\left[\left(\Lambda^{\mathrm{ev}}\pL_{\mL}
^*-\Lambda^{\mathrm{odd}}\pL_{\mL}^*\right)\otimes W_{\xi^0}\otimes 
V_{\tau(w)}\right]^{K_{M^0}}\\
=&\sum_{w\in W^1}(-1)^{\ell(w)+1}\chi(\mL,K_{M^0};W_{\xi^0}\otimes V_{\tau(w)}),
\end{align*}
where $\chi(\mL,K_{M}^0;W_{\xi^0}\otimes V_{\tau(w)})$ denotes the
Euler-characteristic
of the $(\mL,K_{M}^0)$-cohomology with coefficients in the $M^0$-module 
$V_{\tau(w)}\otimes W_{\xi^0}$. Thus the proposition follows from Proposition
\ref{PropDS}, Proposition \ref{propCE}, equation \eqref{infchtheta} and equation
\eqref{cdual}. 
\end{proof}

Next we come to the Plancherel measures. For $\xi\in\hat{M}_d$ we let
$\xi^0\in\hat{M}^0_d$ 
be as above. Fix a regular $\Lambda_{\xi^0}\in i\tL^*$ corresponding to $\xi^0$
as in section \ref{secDS} and let $\Lambda_\xi:=\Lambda_{\xi^0}$.
Choose positive roots $\Delta^+(\mL_\C,\tL_\C;\Lambda_{\xi})$ such that
$\Lambda_\xi$ is 
dominant with respect to these roots. Let $\Delta^+(\gL_\C,\hL_\C;\Lambda_\xi)$
be positive
roots defined via 
$\Delta^+(\mL_\C,\tL_\C;\Lambda_\xi)$ and $e_1$ and let $\rho_{G,\Lambda_\xi}$
be the
half-sum of the
elements 
of $\Delta^+(\mL_\C,\tL_\C;\Lambda_\xi)$. For $\lambda\in\R$ we let
$\mu_\xi(\lambda)$ 
be the Plancherel measure of $\pi_{\xi,\lambda}$.
Then there exists a polynomial
$P_\xi(z)$ 
such that one has
\begin{align}\label{muP}
\mu_{\xi}(\lambda)=P_\xi(i\lambda).
\end{align}
The polynomial $P_\xi(z)$ is given as follows. There exists a constant
$c_{\tilde{X}}$
which depends only on $\tilde{X}$ such that one has
\begin{align}\label{PolyI}
P_\xi(z)=(-1)^nc_{\tilde{X}}\prod_{\alpha\in
\Delta^+(\gL_\C,\hL_\C;\Lambda_\xi)}\frac{\left<\alpha,\Lambda_\xi+z
e_1\right>}{\left<\alpha,\rho_{G,\Lambda_\xi}\right>},
\end{align}
\cite[Theorem 13.11]{Kn1}, \cite[Theorem 13.5.1]{Wa2}.
By \cite[Lemma 5.1]{Ol} and our normalizations one has 
\begin{align}\label{Konstante0}
c_{\tilde{X}}=\frac{1}{|W(A)|\vol(\tilde{X}_d)}.
\end{align}
Note that $P_\xi(z)$ is an even polynomial in $z$. Now let
$w\in W_{\mL}$. We regard
$W_{\mL}$ as a subgroup of $W_{\gL}$. Then if we replace $\Lambda_\xi$ by
$w\Lambda_\xi$,
we have to replace 
$\Delta^+(\gL_\C,\tL_\C;\Lambda_\xi)$ by $w\Delta^+(\gL_\C,\tL_\C;\Lambda_\xi)$.
This implies 
that $P_\xi(z)$ depends only on the $W_{\mL}$-orbit of $\Lambda_\xi$ or
equivalently on the 
infinitesimal character  $\chi(\xi)$ of $\xi$. Thus if for
$\sigma\in\Rep(M^0)$
with 
highest weight $\Lambda(\sigma)$ we let 
\begin{align}\label{Ppoly}
P_\sigma(z):=(-1)^nc_{\tilde{X}}\prod_{\alpha\in
\Delta^+(\gL_\C,\hL_\C)}\frac{\left<\alpha,
\Lambda(\sigma)+\rho_M+z
e_1\right>}{\left<\alpha,\rho_G\right>},
\end{align}
where $c_{\tilde{X}}$ is as in \eqref{PolyI}, it follows that 
$P_\xi(\lambda)= P_\sigma(\lambda)$ if $\chi(\sigma)=\chi(\xi)$. Putting
everything together, we obtain the following corollary. 
\begin{prop}\label{PropTor}
Let $\tau\in\Rep(G)$ and assume that $\tau\not\cong\tau_\theta$. 
Then one has
\begin{align*}
\log{T_X^{(2)}(\tau)}=(-1)^{v}\pi\vol(X)\frac{|W_{\mL}|}{|W_{K_M}|}\sum_{w\in
W^1}(-1)^{\ell(w)}\int_{0}^{|\lambda_{\tau,w}|}P_{\check{\sigma}_{\tau,w}}(t)dt.
\end{align*}

\end{prop}
\begin{proof}
For a given regular and integral $\Lambda\in i\tL^*$ there are exactly
$|W_{\mL}|/|W_{K_M}|$ distinct elements of $\hat{M}_d$ 
with infinitesimal character $\chi_\Lambda$. 
Thus if one combines the Plancherel-Theorem with
Proposition \ref{equcharac},
Proposition \ref{propthetaxi}, 
equation \eqref{muP} and
the previous remarks one obtains 
\begin{align*}
k_t^\tau(1)=(-1)^v\frac{|W_{\mL}|}{|W_{K_M}|}\sum_{w\in
W^1}(-1)^{\ell(w)+1}e^{-t\lambda_{\tau,w}^2}\int_{\R}e^{
-t\lambda^2}P_{\check{\sigma}_{\tau,w}}(i\lambda)d\lambda.
\end{align*}
We let
\begin{align*}
I(t,\tau):=\vol(X)k_t^{\tau}(1).
\end{align*}
By the computations below one has $|\lambda_{\tau,w}|>0$ for every $w\in W^1$.
Thus, since is $P_{\sigma}(\lambda)$
is an even polynomial of degree $2n$ for 
each $\sigma\in\hat{M^0}$, for $s\in\C$ with  $\Real(s)>2n+1$ the integral  
\begin{align*}
\mathcal{M}I(s,\tau):=\int_{0}^\infty t^{s-1}I(t,\tau)dt
\end{align*}
exists. Moreover, by \cite{Fried}, Lemma 2 and Lemma 3, $\mathcal{M}I(s,\tau)$
has a meromorphic continuation to $\C$ which is regular at $0$ and 
if $\mathcal{M}I(\tau)$ denotes its value at $0$ one has
\begin{align*}
\mathcal{M}I(\tau)
=2\pi\vol(X)(-1)^{v}\frac{|W_{\mL}|}{|W_{K_M}|}\sum_{w\in
W^1}(-1)^{\ell(w)}\int_0^{|\lambda_{\tau,w}|}
P_{\check{\sigma}_{\tau,w}}(\lambda)\,d\lambda.
\end{align*}
By definition one has
\begin{align*}
\log{T_X^{(2)}}(\tau)=\frac{1}{2}\mathcal{M}I(\tau)
\end{align*}
and the proposition follows. 
\end{proof}
Now let $G=\SO^0(p,q)$, $p>1$, $p,q$ odd, $p\geq q$, $p=2p_1+1$, $q=2q_1+1$. Let
$n:=p_1+q_1$. 
Let $K=\SO(p)\times\SO(q)$ and $\widetilde
X=G/K$. 
Then $\dim(\widetilde X)=2n+1$. The normalized Killing form is given by
\begin{align*}
\left<X,Y\right>:=\frac{1}{2n-2}B(X,Y).
\end{align*}
We equip $\tilde{X}$ with the Riemannian metric defined by the restriction
of $\left<\cdot,\cdot\right>$ to $\pL$. We have
$\mL\cong\mathfrak{so}(p-1,q-1)$. 
We realize the fundamental Cartan subalgebra as follows.
Let 
\begin{align}
H_1:=E_{p,p+1}+E_{p+1,p}.
\end{align}
Then we put
\begin{align*}
\aL=\R H_1.
\end{align*}
Moreover we let
\begin{equation}\label{basis1}
H_i:=\begin{cases} \sqrt{-1}(E_{2i-3,2i-2}-E_{2i-2,2i-3}),&2\leq i\leq p_1+1\\
\sqrt{-1}(E_{2i-1,2i}-E_{2i,2i-1})&p_1+1< i\leq n+1.\end{cases}
\end{equation}
Then 
\begin{align*}
\tL:=\bigoplus_{i=2}^{n+1}\sqrt{-1}H_i
\end{align*}
is a Cartan subalgebra of $\mL$ and 
\begin{align*}
\hL:=\aL\oplus\tL
\end{align*}
is a Cartan subalgebra of $\gL$. Define 
$e_{i}\in\mathfrak{h}_{\mathbb{C}}^{*}$, $i=1,\dots,n+1$,  by
\begin{align*}
e_{i}(H_{j})=\delta_{i,j},\: 1\leq i,j\leq n+1.
\end{align*}
Then the sets of roots of $(\gL_\C,\hL_\C)$ and $(\mL_\C,
\tL_\C)$ are given by
\begin{align*}
&\Delta(\mathfrak{g}_{\C},\mathfrak{h}_{\mathbb{C}})=\{\pm e_{i}\pm e_{j},\: 1
\leq i<j\leq n+1\}\\
&\Delta(\mathfrak{m}_{\C},\tL_{\C})=\{\pm e_{i}\pm e_{j},\: 2
\leq i<j\leq n+1\}.
\end{align*}
We fix positive systems of roots by
\begin{align*}
&\Delta^{+}(\mathfrak{g}_{\mathbb{C}},\mathfrak{h}_{\mathbb{C}})
:=\{e_{i}+e_{j},\:i\neq j\}\sqcup\{e_{i}-e_{j},\:i<j\}\\ 
& \Delta^{+}(\mathfrak{m}_{\mathbb{C}},\tL_{\mathbb{C}})
:=\{e_{i}+e_{j},\:i\neq j,\:i,j\geq 2\}\sqcup\{e_{i}-e_{j},\:2\leq i<j\}.
\end{align*}
We parametrize the finite-dimensional irreducible representations $\tau$ of $G$
by their highest weights
\begin{equation}\label{Darstellungen von G}
\begin{split}
&\Lambda(\tau)=k_{1}(\tau)e_{1}+\dots+k_{n+1}(\tau)e_{n+1},\:
(k_1(\tau),\dots,k_{n+1}(\tau))\in\mathbb{Z}^{n+1},\\
& k_{1}(\tau)\geq
k_{2}(\tau)
\geq\dots\geq k_{n}(\tau)\geq \left|k_{n+1}(\tau)\right|.
\end{split}
\end{equation}
If $\Lambda$ is a weight as in \eqref{Darstellungen von G}, then
\begin{align}\label{lambdatheta}
\Lambda_{\theta}=k_{1}(\tau)e_{1}+\dots+k_{n}(\tau)e_{n}-k_{n+1}(\tau)e_{
n+1}.
\end{align}
Now we let 
\begin{align}\label{hiweight}
\omega^+_{f,n}:=\sum_{j=1}^{n+1}e_j;\quad\omega^-_{f,n}:=(\omega^+_{f
,n})_\theta=\sum_{j=1}^{n}e_j-e_{n+1}.
\end{align}
Then $\frac{1}{2}\omega^{\pm}_{f,n}$ are the fundamental weights of
$\Delta^+(\gL_\C,\hL_\C)$ which are not invariant under $\theta$. 
We parametrize the  finite-dimensional irreducible representations 
$\sigma$ of $M^0$ by their highest weights
\begin{equation}\label{Darstellungen von M}
\begin{split}
\Lambda(\sigma)=&k_{2}(\sigma)e_{2}+\dots+k_{n+1}(\sigma)e_{n+1},
(k_2(\sigma),\dots,k_{n+1}(\sigma))\in\Z^{n} \\ 
&k_{2}(\sigma)\geq 
k_{3}(\sigma)\geq\dots\geq k_{n}(\sigma)\geq
\left|k_{n+1}(\sigma)\right|\in\mathbb{Z}^{n}.
\end{split}
\end{equation}
For $\sigma\in\Rep(M^0)$ with highest weight $\Lambda(\sigma)$ as in
\eqref{Darstellungen von M} we let
$w_0\sigma\in\Rep(M^0)$ be the representation with highest weight
\begin{align*}
\Lambda(w_0\sigma):=k_{2}(\sigma)e_{2}+\dots+k_n(\sigma)e_n-k_{n+1}(\sigma)e_{
n+1}.
\end{align*}
Then for every $\sigma\in\Rep(M^0)$ one has
$\check{\sigma}=\sigma$ if $n$ is even and $\check{\sigma}=w_0\sigma$ if $n$ is
odd.
Applying equation \eqref{Ppoly} this implies that 
\begin{align}\label{PpolyW}
P_\sigma(\lambda)=P_{w_0\sigma}(\lambda)=P_{\check{\sigma}}(\lambda)
\end{align}
for every $\sigma\in\Rep(M^0)$. 

Let $\tau\in\Rep(G)$ with highest weight $\tau_1e_1+\dots+\tau_{n+1}e_{n+1}$. 
For $k=0,\dots n$ let
\begin{align}\label{lambdatau}
\lambda_{\tau,k}=\tau_{k+1}+n-k
\end{align}
and let $\sigma_{\tau,k}$ be the representation of $G$ with highest weight
\begin{align}\label{sigmatau}
\Lambda_{\sigma_{\tau,k}}:=(\tau_{1}+1)e_{2}+\dots+(\tau_{k}+1)e_{k+1}
+\tau_{k+2}e_{k+2}+\dots+\tau_{n+1}e_{n+1}.
\end{align}
Then as in \cite[section 2.7]{MP} one has
\begin{equation}\label{lambdadecom}
\begin{split}
\{(\lambda_{\tau,w},\sigma_{\tau,w},l(w))\colon w\in W^{1}\}
&=\{(\lambda_{\tau,k},\sigma_{\tau,k},k)\colon k=0,\dots,n\}\\
&\sqcup\{(-\lambda_{\tau,k},w_{0}\sigma_{\tau,k},2n-k)\colon k=0,\dots,n\}.
\end{split}
\end{equation}
Combining \eqref{lambdatheta},  \eqref{PpolyW} and \eqref{lambdadecom} and
Proposition \ref{PropTor} it
follows that  
\begin{align}\label{Torstheta}
T_X^{(2)}(\tau)=T_X^{(2)}(\tau_\theta)
\end{align}
for each $\tau\in\Rep(G)$. Now for $p,q\in\N$ we let $\epsilon(q):=0$ for $q=1$
and 
$\epsilon(q):=1$ for $q>1$ and we let 
\begin{align}\label{const7}
C_{p,q}:=\frac{(-1)^{\frac{pq-1}{2}}2^{\epsilon(q)}\pi}{\vol(\tilde{X}
_d)}\begin{pmatrix}\frac{p+q-2}{2}\\
\frac{p-1}{2} \end{pmatrix} .
\end{align}
Then we have
\begin{prop}\label{L2spin}
For $p,q$, odd, $p\geq q$ let $\widetilde X=\SO^0(p,q)/\SO(p)\times\SO(q)$ and
let
$X=\Gamma\bs \widetilde X$. Let $\Lambda\in\hL^*_\C$ be a highest weight 
as in \eqref{Darstellungen von G} and assume that $\Lambda_\theta\neq\Lambda$. 
For $m\in\N$ let $\tau_\Lambda(m)$ be the 
irreducible representation of $\SO^0(p,q)$ with highest
weight $m\Lambda$. 
There exists a polynomial $P_\Lambda(m)$ whose coefficients depend only on
 $\Lambda$, such that for all $m\in\N$ we have
\begin{align*}
\log{T_X^{(2)}(\tau_\Lambda(m))}=C_{p.q}\vol(X)P_\Lambda(m).
\end{align*}
Moreover there is a constant $C_\Lambda>0$, which depends on $\Lambda$, 
such that
\begin{align}\label{eqP}
P_\Lambda(m)=C_\Lambda\cdot
m\dim(\tau_\Lambda(m))+O\left(\dim(\tau_\Lambda(m))\right)
\end{align}
as $m\to\infty$. If 
$\Lambda=\omega_{f,n}^\pm$, where $\omega_{f,n}^\pm$ are as in
\eqref{hiweight}, then $C_\Lambda=1$. 
\end{prop}
\begin{proof}
Let $\Lambda=\tau_1e_1+\cdots+\tau_{n+1}e_{n+1}$. By \eqref{lambdatheta} and
\eqref{Torstheta} we may assume that $\tau_{n+1}>0$.
Put $\tau(m):=\tau_\Lambda(m)$. Then
\begin{equation}\label{lambdam}
\lambda_{\tau(m),k}=m\tau_{k+1}+n-k,\quad k=0,\dots,n,
\end{equation}
and by Proposition \ref{PropTor},\eqref{lambdadecom} and \eqref{PpolyW} we have
\begin{align*}
&\log T_X^{(2)}(\tau(m))\\
=&2\pi\vol(X)(-1)^{v}\frac{|W_{\mL}|}{|W_{K_M}|}\sum_{k=0}
^n(-1)^k\int_0^{\lambda_{\tau(m),k}}
P_{\sigma_{\tau(m),k}}(t)\,dt.
\end{align*}
In the hyperbolic case the term $(-1)^{v}|W_{\mL}|/|W_{K_M}|$ equals 1. 
Therefore this equation agrees with \cite[(5.16), (5.17)]{MP}. Note that
$n=\dim\nL$. Let $c_{\widetilde X}$ be defined by \eqref{Konstante0} and put
\begin{align}\label{polynom4}
P_\Lambda(m):=\frac{(-1)^{n}}{c_{\tilde{X}}}\sum_{k=0}
^n(-1)^k\int_0^{\lambda_{\tau(m),k}} P_{\sigma_{\tau(m),k}}(t)\,dt.
\end{align}
Then it follows from \eqref{Ppoly} and \eqref{lambdadecom} that $P_\Lambda$ is a
polynomial in $m$ whose coefficients depend only 
on $\Lambda$. By definition one has
\begin{align*}
\log T_X^{(2)}(\tau(m))=2\pi\vol(X)(-1)^{v+n}
\frac{|W_{\mL}|}{|W_{K_M}|}c_{\tilde{X}}P_\Lambda(m).
\end{align*}
So it remains to compute the constant. By \eqref{Konstante0} and Lemma 
\ref{LemWA} one has
\begin{align*}
\frac{|W_{\mL}|}{|W_{K_M}|}
c_{\tilde{X}}=\frac{|W_{\mL}|}{\left|W_{\kL_{\mL}}\right|}\frac{1}{
2\vol(\tilde{X}_d)}.
\end{align*}
Recall that $\mL_\C\cong \soL(2n,\C)$,
$(\kL_{\mL})_\C\cong\soL(2p_1,\C)\oplus\soL(2q_1,\C)$ and so by \cite[page
685]{Kn2} one has $|W_{\mL}|=n!2^{n-1}$,
$\left|W_{\kL_{\mL}}\right|=p_1!q_1!2^{n-1-\epsilon(q)}$, where 
$\epsilon(q)$ is as above. Thus 
as in \cite[Proposition 1.3]{Ol} one has
\begin{align*}
\frac{|W_{\mL}|}{\left|W_{\kL_{\mL}}\right|}=2^{\epsilon(q)}\begin{pmatrix}\frac
{p+q-2}{2}\\
\frac{p-1}{2} \end{pmatrix} .
\end{align*}
Furthermore one has $v=\frac{\dim\pL_{\mL}}{2}=\frac{(p-1)(q-1)}{2}$ and thus 
we get $v+n=\frac{pq-1}{2}$.
This proves the first part of the proposition.

To determine the highest order term of the polynomial $P_\Lambda(m)$, we
proceed 
as in \cite[Lemma 5.4]{MP} to show that
\[
 P_{\sigma_{\tau(m),k}}(t)=(-1)^{n+k} c_{\widetilde X}\dim(\tau(m))
\prod_{\substack{j=0\\ j\neq k}}^{n}\frac{t^{2}
-\lambda_{\tau(m),j}^{2}}{\lambda_{\tau(m),k}^{2}-\lambda_{\tau(m),j}^{2}}.
\]
Denote the product on the right by $\Pi_k(t;m)$. Then it follows from 
\eqref{polynom4} that
\begin{equation}\label{polynom5}
P_\Lambda(m)=\dim(\tau(m))\cdot\sum_{k=0}^n\int_0^{\lambda_{\tau(m),k}}
\Pi_k(t;m)\,dt.
\end{equation}
To deal with the sum, we follow \cite[5.9.1]{BV}. Put $\lambda_{\tau(m),n+1}=0$.
Then $\lambda_{\tau(m),k}$,
$k=0,\dots,n+1$ is a strictly decreasing sequence. For $k=0,\dots,n$ set
\[
Q_k(t;m):=\sum_{j=0}^k\Pi_j(t;m).
\]
Then $Q_k(t;m)$ is the unique even polynomial of degree $\le 2n$ which satisfies
\begin{equation}\label{root}
Q_k(\pm\lambda_{\tau(m),j})=\begin{cases}1,&\mathrm{if}\,\,j\le k,\\
0,&\mathrm{if}\,\, n\geq j>k.
\end{cases}
\end{equation}
Moreover we have
\begin{equation}\label{equat1}
\sum_{k=0}^n\int_0^{\lambda_{\tau(m),k}}\Pi_k(t;m)\,dt=\sum_{k=0}^n
\int_{\lambda_{\tau(m),k+1}}^{\lambda_{\tau(m),k}}Q_k(t;m)\,dt.
\end{equation}
As proved in \cite[Sect. 5.9.1]{BV}, each integral on the right is 
positive. This can be seen as follows. By \eqref{root}, the polynomial 
$Q_k^\prime$ has a root 
in each interval
$[\lambda_{\sigma_{\tau(m),j+1}},\lambda_{\sigma_{\tau(m),j}}]$,
$[-\lambda_{\sigma_{\tau(m),j}},-\lambda_{\sigma_{\tau(m),j+1}}]$ for $1\le
j<n$, 
$j\neq k$ and a 
root in $[-\lambda_{\sigma_{\tau(m),n}},\lambda_{\sigma_{\tau(m),n}}]$. Since
$Q_k^\prime$
is of degree $\le 2n-1$, it follows that $Q_k$ is either constant or strictly
increasing on $[\lambda_{\sigma_{\tau(m),k+1}},\lambda_{\sigma_{\tau(m),k}}]$.
Furthermore, $Q_n(t;m)$ is a polynomial of degree $2n$, which is equal to 1 
at $2n+2$ pairwise distinct points. Hence $Q_n\equiv 1$. Thus by 
\eqref{lambdam} and \eqref{equat1} we get
\begin{align}\label{inequ}
(n+1)(m\tau_1+n)=(n+1)\lambda_{\tau(m),0}&\geq\sum_{k=0}^n(\lambda_{\tau(m),k}
-\lambda_{\tau(m),k+1})\nonumber\\ &\geq\sum_{k=0}^n\int_0^{
\lambda_{\tau(m),k}}\Pi_k(t;m)\,dt\ge \tau_{n+1}m.
\end{align}
Since $P_\Lambda(m)$ is a polynomial in $m$, it follows that there exists
$C_\Lambda\ge 
\tau_{n+1}>0$ such that \eqref{eqP} holds . If $\Lambda$ is one of the
fundamental weight $\omega_{f,n}^\pm$,
defined by \eqref{hiweight}, then it follows as in \cite[Section 5]{MP} that
$C_\Lambda=1$. This proves the second part of the proposition. 
\end{proof}

Finally we turn to the case $G=\SL_3(\R)$, $K=\SO(3)$. We define our 
fundamental Cartan subalgebra as follows. Let
\begin{align*}
H_1:=\diag(1,1,-2);\quad \aL:= \R H_1.
\end{align*}
Then we have $\mL=\slL_2(\R)$, if $\slL_2(\R)$ is embedded into $\gL$ as an
upper left block.
Let 
\begin{align*}
H_2:=\begin{pmatrix}0&1\\-1&0\end{pmatrix},\quad \tL:=\R T_1
\end{align*}
embedded into $\gL$ as an upper left block.
Then $\tL$ is a Cartan subalgebra of $\mL$ and
\begin{align}\label{cartan}
\hL:=\aL\oplus\tL
\end{align}
is a $\theta$-stable fundamental Cartan subalgebra of $\gL$. Note that $\hL$ is
different from 
the usual Cartan subalgebra $ \tilde{\hL}$ of $\gL$ which consist of all
diagonal matrices of trace $0$.  
Define $f_1\in\aL^*$ and $f_2\in i\tL^*$ by
\begin{align*}
f_1(H_1)=3; \quad f_2(H_2)=i.
\end{align*}
We fix $f_1$ as a positive restricted root of $\aL$. Then we can define positive
roots by
\begin{align*}
\Delta^+(\gL_\C,\hL_\C):=\{f_1-f_2,\:f_1+f_2,\: 2f_2\};\quad
\Delta^+(\mL_\C,\tL_\C)=\{2f_2\}.
\end{align*}
Under our normalization one has 
\begin{align}\label{inProd}
\left<f_1,f_1\right>=1;\quad \left<f_2,f_2\right>=\frac{1}{3};\quad
\left<f_1,f_2\right>=0.
\end{align} 
One easily sees that $\dim\nL=2$, hence $n=1$. Moreover by 
\cite[page 485]{Kn2} one has $|W(A)|=1$.
For $k\in\N$ let $\sigma_k\in\Rep(M^0)$ be of 
highest weight $kf_2$. Then it follows from \eqref{Ppoly} and 
\eqref{Konstante0} that
\begin{align}\label{Ppolysldrei}
P_{\sigma_k}(z)=-\frac{9}{8\vol(\tilde{X}_d)}\:(k+1)\left(z^2-\left(\frac
{k+1}{3}
\right)^2\right).
\end{align}
Define $e_i\in \tilde{\hL}_\C^*$ by
$e_i(\diag(t_1,t_2,t_3))=\sum_j\delta_{i,j}t_j$.
Then one can choose positive roots 
\begin{align}\label{roots6}
\Delta^+(\gL_\C, \tilde{\hL}_\C):=\{e_1-e_2,\:e_1-e_3,\:e_2-e_3\}
\end{align}
and there is a standard inner-automorphism $\Phi$ of $\gL_\C$ which sends
$\hL_\C$ to $\tilde{\hL}_\C$ and which satisfies 
\begin{align}\label{isom6}
\Phi^*(e_1-e_2)=2f_2;\quad \Phi^*(e_1-e_3)= f_1+f_2;\quad \Phi^*(e_2-e_3)=
f_1-f_2.
\end{align}
The fundamental weights
$\tilde{\omega}_1,\tilde{\omega}_2\in\tilde{\hL}_\C^*$  
are given by
\begin{align*}
\tilde{\omega}_1=\frac{2}{3}(e_1-e_2)+\frac{1}{3}(e_2-e_3)
\end{align*}
and
\begin{align*}
\tilde{\omega}_2=\frac{1}{3}(e_1-e_2)+\frac{2}{3}(e_2-e_3).
\end{align*}
Thus the fundamental weights $\omega_1,\omega_2\in\hL_\C^*$ are given by
\begin{align}\label{fundwei1}
\omega_1:=\Phi^*(\tilde{\omega}_1)=\frac{1}{3}f_1+f_2;\quad
\omega_2:=\Phi^*(\tilde{\omega}_2)=\frac{2}{3}f_1.
\end{align}
Let $\N_0:=\N\sqcup\{0\}$. If $\Lambda$ is a weight, $\Lambda=
\tau_1\omega_1+\tau_2\omega_2$,
$\tau_1,\tau_2\in\N_0$, then a standard computation 
shows that
\begin{align}\label{Lambdatheta2}
\Lambda_{\theta}=\tau_2\omega_1+\tau_1\omega_2.
\end{align}
Now we fix $\tau_1, \tau_2\in\N_0$, $\tau_1+\tau_2>0$ 
and for $m\in\N$ we let 
$\tau(m)$ be the representation of $G$ with highest weight 
\begin{align}\label{highweight}
\Lambda(\tau(m)):=m\tau_1\omega_1+m\tau_2\omega_2.
\end{align}
We let
$\tilde{W}_{\gL}$ be the Weyl-group of $\Delta(\gL_\C,\tilde{\hL}_\C)$. 
Then $\tilde{W}_{\gL}$ consists of all permutations of $e_1,e_2,e_3$.
Let
\begin{align*}
\tilde{W}^1:=(\Phi^*)^{-1}W^1=\{w\in \tilde{W}_{\gL}\colon w^{-1}(e_1-e_2)>0\}.
\end{align*}
Then one has 
\begin{align*}
&\{(w,\ell(w)); w\in
\tilde{W}^1\} \\ =&\biggl\{(\Id,0);\:\biggl(\begin{pmatrix}e_1&e_2&e_3\\ e_1&
e_3&e_2\end{pmatrix},1\biggr);\:\biggl(\begin{pmatrix}e_1&e_2&e_3\\
e_3&e_1&e_2\end{pmatrix},2\biggr)\biggr\}.
\end{align*}
By a direct computation we get
\begin{equation}\label{wtilde}
\begin{split}
&\{w(\Lambda(\tau(m))+\tilde{\rho}_{G}),\ell(w);w\in
\tilde{W}^1\}\\
&=\biggl\{\left(\frac{2m\tau_1+m\tau_2+3}{3}(e_1-e_2)
+\frac{m\tau_1+2m\tau_2+3}{3}(e_2-e_3);0\right),\\
&\hskip1truecm\left(\frac{2m\tau_1+m\tau_2+3}{3}(e_1-e_2)
+\frac{m\tau_1-m\tau_2}{3}(e_2-e_3);1\right),\\
&\hskip1truecm\left(\frac{-m\tau_1+m\tau_2}{3}(e_1-e_2)+\frac{
-2m\tau_1-m\tau_2-3}{3}
(e_2-e_3);2\right)\biggr\}.
\end{split}
\end{equation}
As in \cite[5.9.2]{BV} we introduce the following constants 
\begin{equation}\label{Glsldrei1}
\begin{split}
&A_{1}(\tau(m)):=\frac{m\tau_1+1}{2};\:
A_2(\tau(m)):=\frac{m\tau_1+m\tau_2+2}{2};\\
&A_3(\tau(m)):=\frac{m\tau_2+1}{2}
\end{split}
\end{equation}
and 
\begin{equation}\label{Glsldrei2}
\begin{split}
&C_1(\tau(m)):=\frac{m\tau_1+2m\tau_2+3}{3};\:
C_2(\tau):=\frac{m\tau_1-m\tau_2}{3}; \\
&C_3(\tau):=\frac{2m\tau_1+m\tau_2+3}{3}.
\end{split}
\end{equation}
Note that on $\tilde{\hL}_\C^*$ one has 
$\tilde{\omega}_1=e_1;\quad \tilde{\omega}_2=e_1+e_2$,
since the matrices in $\tilde{\hL}_\C^*$ have trace $0$.
Then, combining \eqref{isom6} and \eqref{wtilde}, we get
\begin{align*}
\begin{split}
&\{\left(\Lambda(\sigma_{\tau(m),w}),\lambda_{\tau(m),w},\ell(w)\right);w\in
W^1\} =\\ &\{\left((2A_{1}(\tau(m))-1)f_2,C_{1}(\tau(m)),0\right),
\left((2A_{2}(\tau(m))-1)f_2,C_{2}(\tau(m)),1\right),
\\ &\left((2A_{3}(\tau(m))-1)f_2,-C_{3}(\tau(m)),2\right)\}.
\end{split}
\end{align*}
Thus if we apply \eqref{Ppolysldrei} we obtain 
\begin{align}\label{Glsldrei3}
&\sum_{w\in
W^{1}}(-1)^{\ell(w)}\int_{0}^{|\lambda_{\tau(m),w}|}P_{\sigma_{\tau(m),w}}
(t)dt\nonumber\\
=&-\frac{C_{\SL_3(\R)}}{\vol(\tilde{X}_d)}\sum_{k=1}^{3}(-1)^{k+1}A_{k}
(\tau(m))\int_{0}^{
\left|C_k(\tau(m))\right|}
\left(\frac{9}{4}t^2-A_{k}(\tau(m))^2\right)dt\nonumber\\
=&-\sum_{k=1}^3(-1)^{k+1}\frac{
A_{k}(\tau(m))|C_{k}(\tau(m))|}{4\vol(\tilde{X}_d)}
\left(3C_{k}(\tau(m))^2-4A_{k}(\tau(m))^2\right).
\end{align}
We can now prove our main result about the $L^2$-torsion for the case
$G=SL_3(\R)$.
\begin{prop}\label{L2SL3}
Let $\widetilde X=\SL(3,\R)/\SO(3)$ and $X=\Gamma\bs \widetilde X$. 
Let  $\Lambda\in\hL^*_\C$ be a highest weight with $\Lambda_\theta\neq\Lambda$.
For $m\in\N$ let $\tau_\Lambda(m)$ be the irreducible representation of 
$\SL(3,\R)$ with highest weight $m\Lambda$. 
There exists a polynomial $P_\Lambda$ whose coefficients depend only on
$\Lambda$ such that 
\begin{align*}
\log T_X^{(2)}(\tau_{\Lambda}(m))=\frac{\pi\vol(X)}{\vol(\widetilde X_d)}
P_\Lambda(m).
\end{align*}
Moreover, there exists a constant $C(\Lambda)>0$ depending only on $\Lambda$
such that 
\begin{align*}
P_\Lambda(m)=C(\Lambda)m\dim(\tau_{\Lambda}(m))+O(\dim(\tau_{\Lambda}(m))), 
\end{align*}
as $m\to\infty$. If $\Lambda$ equals one of the fundamental weights 
$\omega_{f,i}$ then $C(\Lambda)=4/9$.
\end{prop}
\begin{proof}
There exist $\tau_1,\tau_2\in\N_0$, $\tau_1\neq \tau_2$, such that
$\Lambda=\tau_1\omega_1+\tau_2\omega_2$. 
Put $\tau(m):=\tau_\Lambda(m)$. Then by Proposition \ref{PropTor}, equation
\eqref{Glsldrei1}, \eqref {Glsldrei2} and \eqref{Glsldrei3}, 
the first statement is proved and it remains to consider the asymptotic
behavior of the polynomial $P_\Lambda$.
We differentiate two cases. First we assume that $\tau_1\tau_2\neq 0$.
Then if we put
\begin{align*}
\alpha_4(\tau):=
\begin{cases}
-\frac{\tau_2^4}{18}+\frac{2\tau_1^3\tau_2}{9}+\frac{\tau_1^2\tau_2^2}{3}; &
\tau_1\geq
\tau_2\\-\frac{\tau_1^4}{18}+\frac{2\tau_2^3\tau_1}{9}+\frac{\tau_1^2\tau_2^2}{3
}; &
\tau_2\geq \tau_1,
\end{cases}
\end{align*}
an explicit computation using equation
\eqref{Glsldrei1}, \eqref {Glsldrei2} and \eqref{Glsldrei3} shows that 
\begin{align*}
\sum_{w\in
W^{1}}(-1)^{\ell(w)}\int_{0}^{|\lambda_{\tau(m),w}|}P_{\sigma_{\tau(m),w}}
(t)dt=-\frac{\alpha_4(\tau)}{\vol(\tilde{X}_d)} m^4+O(m^3),
\end{align*}
as $m\to\infty$. Note that $\alpha_4(\tau)>0$  by our
assumption on $\tau_1$ and $\tau_2$. 
Now we assume that $\tau_1\tau_2=0$.
Then if we define
\begin{align*}
\alpha_3(\tau):=\frac{2(\tau_1^3+\tau_2^3)}{9},
\end{align*}
an explicit computation using equation
\eqref{Glsldrei1}, \eqref {Glsldrei2} and \eqref{Glsldrei3} gives
\begin{align*}
\sum_{w\in
W^{1}}(-1)^{\ell(w)}\int_{0}^{\lambda_{\tau(m),w}}P_{\sigma_{\tau(m),w}}
(t)dt=-\frac{\alpha_3(\tau)}{\vol(\tilde{X}_d)} m^3+O(m^2),
\end{align*}
as $m\to\infty$. 
For $\SL_3(\R)$ one has $v=1$ and using Lemma \ref{LemWA} one gets
$\frac{|W_{\mL}|}{|W_{K_M}|}=1$. Moreover, every element of $\Rep(M^0)$ is
self-dual. Thus using
Proposition \ref{PropTor} we obtain 
\begin{align*}
\log T_X^{(2)}(\tau(m))=\vol(X)\frac{\pi \alpha_4(\tau)}
{\vol(\tilde{X}_d)}m^4+O(m^3) 
\end{align*}
as $m\to\infty$, if $\tau_1\tau_2\neq 0$,
and 
\begin{align*}
\log T_X^{(2)}(\tau(m)=\vol(X)\frac{\pi\alpha_3(\tau)}{\vol(
\tilde{X}_d)}m^3+O(m^2),
\end{align*}
as $m\to \infty$, if $\tau_1\tau_2=0$. Now we define constants 
\begin{align*}
d_3(\tau):=\frac{\tau_1^2\tau_2+\tau_2^2\tau_1}{2};\quad
d_2(\tau):=\left(\frac{
4\tau_1\tau_2+\tau_1^2+\tau_2^2}{2}\right).
\end{align*}
Then by Weyl's dimension formula one has
\begin{align*}
\dim\tau(m)=d_3(\tau)m^3+d_2(\tau)m^2+O(m),
\end{align*}
as $m\to\infty$. Note that $d_3(\tau)>0$ for $\tau_1\tau_2\neq 0$ and that
$d_3(\tau)=0$, $d_2(\tau)>0$ 
for $\tau_1\tau_2=0$. This completes the proof of the proposition.
\end{proof}

\section{Lower bounds of the spectrum}\label{seclowbd}
\setcounter{equation}{0}
In this section we assume that $\widetilde{X}$ 
is odd-dimensional and that $\delta(\widetilde X)=1$. Our goal is to
establish the lower bound \eqref{lowerbd0} for the spectrum of the Laplace 
operators $\Delta_p(\tau_\lambda(m))$. To this end we use \eqref{bochhodge1},
which reduces the problem to the estimation from below of the endomorphism 
$E_p(\tau_\lambda(m))$. \\
First we introduce some notation.
Let $\widetilde X=G/K$. 
There is a decomposition $\widetilde X=\widetilde X_0\times\widetilde X_1$ with 
$\delta(\widetilde X_0)=0$ and $\widetilde X_1$ is an irreducible symmetric
space with $\delta(\widetilde X_1)=1$. Since $\widetilde X_0$ is
even-dimensional, the dimension of
$\widetilde X_1$ is odd. Let $G=G_0\times G_1$ be the corresponding
decomposition of $G$. Then $\delta(G_0)=0$ and by the classification of simple
Lie groups, 
$G_1=\SO^0(p,q)$, $p,q$ odd,
or $G_1=\SL(3,\R)$. Let $\gL_i$, $i=0,1$ be the Lie algebra of $G_i$. Let
$\tL_0\subset \gL_0$ be a compact Cartan subalgebra and let $\hL_1\subset\gL_1$
be a fundamental Cartan subalgebra. Then $\hL_1$ is of split rank one. Put
\[
\hL:=\tL_0\oplus\hL_1.
\] 
Then $\hL$ is a Cartan subalgebra of split rank one. 
Let $(\tau,V_\tau)\in\Rep(G)$ with highest weight $\lambda\in\hL^*_\C$.
Then $\lambda=\lambda_0+\lambda_1$, where $\lambda_0\in\tL_{0,\C}^*$ and 
$\lambda_1\in\hL_{1,\C}^*$ are highest weights. Let $\theta\colon\gL\to\gL$ be 
the Cartan involution. Assume
that $\lambda_\theta\neq\lambda$. Then $\lambda_1$ satisfies $(\lambda_1)_\theta
\neq\lambda_1$. Let $(\tau_i,V_{\tau_i})\in\Rep(G_i)$, $i=0,1$, be the 
representations with highest weight $\lambda_i$. Then 
$\tau\cong\tau_0\otimes\tau_1$. Let 
\[
\gL_i=\kL_i\oplus\pL_i
\]
be the Cartan decomposition of $\gL_i$, $i=0,1$. We may choose $\pL$ such that
$\pL=\pL_0\oplus\pL_1$. Then we have
\[
\Lambda^p\pL^*\otimes V_\tau\cong\bigoplus_{r+s=p}\left(\Lambda^r\pL_0^*\otimes
V_{\tau_0}\right)\otimes \left(\Lambda^s\pL_1^*\otimes V_{\tau_1}\right)
\]
Let $\Omega_i
\in{\mathcal Z}(\gL_{i,\C})$, $i=1,2$, be the Casimir operator of $\gL_i$. Then
$\Omega=\Omega_0\otimes\Id + \Id\otimes\Omega_1$. Similarly, we have
$\Omega_K=\Omega_{0,K}\otimes\Id + \Id\otimes\Omega_{1,K}$.
Set
\[
\nu_{i,p}(\tau_i):=\Lambda^p\Ad^*_{\pL_i}\otimes \tau_i\colon K_i\to 
\GL(\Lambda^p\pL_i^*\otimes V_{\tau_i}), \quad i=0,1.
\]
Let
\begin{equation}\label{endoi}
E_{i,p}(\tau_i):=\tau_i(\Omega_i)\Id_i-\nu_p(\tau_i)(\Omega_{i,K}),\quad i=0,1.
\end{equation}
be the corresponding endomorphisms acting in $\Lambda^p\pL_i^*\otimes
V_{\tau_i}$. Then it follows that
\begin{equation}\label{decomp3}
E_p(\tau)=\bigoplus_{r+s=p}\left(E_{0,r}(\tau_0)\otimes \Id + \Id\otimes
E_{1,s}(\tau_1)\right).
\end{equation}
Therefore it suffices to estimate $E_{i,p}(\tau_i)$, $i=0,1$.\\
Let us first recall the general formula for the Casimir eigenvalues.  We let
$\gL$ be a semisimple real Lie algebra with Cartan decomposition
$\gL=\kL\oplus\pL$. Let $\tL$ be a Cartan subalgebra of $\kL$ and let
$\hL=\tL\oplus \bL$, $\bL\subset \pL$ be a $\theta$-stable
Cartan subalgebra of $\gL$ containing $\tL$. Let the associated groups $G$ and
$K$ be as in the
introduction. Let $\left\|\cdot\right\|$ denote the norm induced by the 
(suitably normalized)
Killing form 
on the real vector space $\bL^*\oplus i\tL^*$. 
Fix positive roots $\Delta^+(\gL_\C,\hL_\C)$, $\Delta^+(\kL_\C,\tL_\C)$
and let 
$\rho_G$ resp. $\rho_K$ be the half sum of the positive roots.
Let $\tau$ be an irreducible finite-dimensional complex representation of $G$
with 
highest weight $\Lambda(\tau)\in\bL^*\oplus i\tL^*$ and let $\nu$ be an
irreducible 
unitary representation of $K$ with highest weight
$\Lambda(\nu)\in i\tL^*$.
Then we have 
\begin{align}\label{casimir}
\tau(\Omega)=\|\Lambda(\tau)+\rho_G\|^2-\|\rho_G\|^2; \quad
\nu(\Omega_K)=\|\Lambda(\nu)+\rho_K\|^2-\|\rho_K\|^2.
\end{align}
Now we prove the following general bound, which we use to deal with
$E_{0,p}(\tau_0)$.
\begin{lem}\label{genbd}
Let $\lambda\in\hL^*_\C$ be a highest weight. Given $m\in\N$, let 
$\tau_\lambda(m)$ be the irreducible representation with highest weight 
$m\lambda$. There exists $C>0$ such that 
\[
E_p(\tau_\lambda(m))\ge -Cm
\]
for all $p=0,\dots,d$ and $m\in\N$. 
\end{lem}
\begin{proof}
Let $\tau\in\Rep(G)$ be of highest weight $\Lambda(\tau)$. Let
$\nu^\prime\in\hat K$ with highest weight $\Lambda(\nu^\prime)\in i\tL^*$.
Assume 
that $[\tau|_K\colon\nu^\prime]\neq 0$. We claim that there is a weight 
$\lambda$ of 
$\tau$ such that $\Lambda(\nu^\prime)=\lambda|_{\tL}$. To see this, let 
$V_\tau$ be the 
space of the representation $\tau$ and let $V_\tau(\Lambda(\nu^\prime))$ be the 
eigenspace of $\tL$ with eigenvalue $\Lambda(\nu^\prime)$. Then 
$V_\tau(\Lambda(\nu^\prime))$ is invariant under $\hL$. So it decomposes into
joint eigenspaces of $\hL$. Let $\lambda$ be the weight of 
one of these eigenspaces. Then $\lambda|_{\tL}=\Lambda(\nu^\prime)$. Now we 
note that as a 
weight of $\tau$, $\lambda$ belongs to the convex hull of the Weyl group 
orbit of $\Lambda(\tau)$ (see \cite[Theorem 7.41]{Ha}). Thus we get
\begin{equation}\label{lambdaest}
\|\Lambda(\tau)\|\ge \|\lambda\| \ge \|\lambda|_\tL\|
=\|\Lambda(\nu^\prime)\|.
\end{equation}
Now let $\nu\in \widehat K$ with $[\nu_p(\tau)\colon\nu]\neq0$. Then by
\cite[Proposition 9.72]{Kn2} there exists $\nu^\prime\in\widehat K$ with
$[\tau|_K\colon\nu^\prime]\neq0$ of highest weight $\Lambda(\nu^\prime)\in 
i\tL^*$ and $\mu\in i\tL^*$  which is a weight of $\nu_p$ such that
the
highest weight $\Lambda(\nu)$ of $\nu$ is given by $\mu+\Lambda(\nu^\prime)$.
Since 
$\Lambda(\tau)$ is dominant we have 
\begin{align*}
\left\|\Lambda(\tau)+\rho_G\right\|^2\geq \left\|\Lambda(\tau)\right\|^2.
\end{align*}
Thus by  \eqref{lambdaest} we get 
\begin{align*}
&\|\Lambda(\tau)+\rho_G\|^2-\|\Lambda(\nu)+\rho_K\|^2\\&\ge 
\|\Lambda(\tau)\|^2-\|\Lambda(\nu^\prime)\|^2
-2\|\mu+\rho_K\|\cdot \|\Lambda(\nu^\prime)\|-\|\mu+\rho_K\|^2\\
&\ge -2\|\mu+\rho_K\|\cdot\|\Lambda(\tau)\|-\|\mu+\rho_K\|^2.
\end{align*}
There is $C>0$ such that $\|\mu+\rho_K\|\le C$ for all weights $\mu$ of
$\nu_p$.
Hence there is $C_1>0$ such that for all $\tau\in\Rep(G)$ one has
\begin{equation}\label{inequ1}
\|\Lambda(\tau)+\rho_G\|^2-\|\Lambda(\nu)+\rho_K\|^2\ge 
-C_1(\|\Lambda(\tau)\|+1)
\end{equation}
for all $\nu\in\widehat K$ with $[\nu_p(\tau)\colon\nu]\neq0$. Now we apply this
to $\tau_\lambda(m)$. By definition of $\tau_\lambda(m)$ we have 
$\Lambda(\tau_\lambda(m))=m\lambda$.
Using \eqref{inequ1}, \eqref{casimir},
the lemma follows.
\end{proof}
Now we turn to the estimation of $E_{1,p}(\tau_1)$. In this case we have 
either $G_1=\SO^0(p,q)$, $p,q$ odd, or $G=\SL(3,\R)$. We
deal with these
cases separately. 

\subsection{The case $G=\SO^0(p,q)$.} Let $p=2p_1+1$,
$q=2q_1+1$. 
Let $n:=p_1+q_1$. Let $K=\SO(p)\times\SO(q)$
and $\widetilde X=G/K$. We let $\tL$ and $\hL$ be as in section
\ref{secdelta1}. Also the Killing form will be normalized as in this section.
Then we have the following lemma.
\begin{lem}\label{lemconv}
Let $\Lambda\in\hL_C^*$ be given as  $\Lambda=k_1e_1+\dots+k_{n+1}e_{n+1}$,
$k_1\geq k_2\geq \dots \geq k_{n+1}\geq 0$.
Let $\Lambda'\in\hL_\C^*$ belong to the convex hull of the set
$\{w\Lambda,\:w\in W_G\}$ and let 
$\lambda\in i\tL^*$ be given by $\lambda:=\Lambda'|_{\tL}$. Then one has 
\begin{align*}
\left\|\lambda\right\|^2\leq \sum_{i=1}^{n}k_i^2.
\end{align*}
\end{lem}
\begin{proof}
Recall that the Weyl group $W_G$ consist of permutations and even sign changes 
of the $e_1,\dots,e_{n+1}$. Thus there exist 
$\alpha_1,\dots,\alpha_m \in (0,1)$, $\sum_{j=1}^{m}\alpha_j=1$, and for each
$j=1,\dots,m$ 
a $\sigma_j\in S^{n+1}$, the symmetric group, and a sequence
$\epsilon_{j,1},\dots,\epsilon_{j,n+1}\in\{\pm 1 \}$ 
such that 
\begin{align*}
\Lambda'=\sum_{j=1}^{m}\alpha_{j}\left(\sum_{i=1}^{n+1}\epsilon_{j,i}k_ie_{
\sigma_j(i)}\right).
\end{align*}
Thus one has
\begin{align*}
\lambda=\sum_{j=1}^{m}\alpha_j\left(\sum_{\substack{i=1\\ \sigma_j(i)\neq
p_1+1}}^{n+1}\epsilon_{j,i}k_ie_{\sigma_j(i)}\right)
\end{align*}
and so one gets
\begin{align*}
\left\|\lambda\right\|\leq &\sum_{j=1}^{m}\alpha_j\left\|\sum_{\substack{i=1\\
\sigma_j(i)\neq
p_1+1}}^{n+1}\epsilon_{j,i}k_ie_{\sigma_j(i)}\right\|
= \sum_{j=1}^{m}\alpha_j\sqrt{\sum_{\substack{i=1\\
\sigma_j(i)\neq p_1+1}}^{n+1}k_i^2}\\ &\leq \sum_{j=1}^{m}\alpha_j
\sqrt{\sum_{i=1}^{n} k_i^2}=\sqrt{\sum_{i=1}^{n} k_i^2}.
\end{align*}
\end{proof}
Now we let $\Lambda(\tau)\in\hL_\C^*$ be given by 
\begin{align*}
\Lambda(\tau):=\tau_1e_1+\dots+\tau_{n+1}e_{n+1}, \quad \tau_1\geq \tau_2\geq
\dots\geq \tau_{n+1}>0. 
\end{align*}
For $m\in\N$ we let $\tau(m)$ be the representation of $G$ with highest weight
\begin{align*}
\Lambda(\tau(m)):=m\Lambda(\tau).
\end{align*}
Then we have the following proposition.
\begin{prop}\label{lowerbd5}
There exists a constant $C$ such that 
\begin{align*}
E_p(\tau(m))\geq m^2\tau_{n+1}-Cm
\end{align*}
for all $m$.
\end{prop}
\begin{proof}
Recall that $\nu_p(\tau(m))=\tau(m)|_{K}\otimes\nu_p$. 
Let $\nu\in\hat{K}$ be such that $\left[\nu_p(\tau(m)):\nu\right]\neq 0$. 
By \cite[Proposition 9.72]{Kn2}, there exists a $\nu^\prime\in\hat{K}$ with 
$\left[\tau(m):\nu^\prime\right]\neq 0$ of 
highest weight $\lambda(\nu^\prime)\in\mathfrak{b}_{\C}^*$ and a
$\mu\in\mathfrak{b}_{\C}^*$
which is a weight of $\nu_p$ such that the highest weight $\lambda(\nu)$ of
$\nu$ is given by $\mu+\lambda(\nu^\prime)$. Now there is a
$\widetilde{\Lambda}\in\hL_{\C}^*$ 
which is a weight of $\tau(m)$ such that
$\lambda(\nu')=\widetilde{\Lambda}|_{\tL}$. By \cite[Theorem 7.41]{Ha}, 
$\widetilde{\Lambda}$ belongs to the convex hull of the Weyl group 
orbit of $\Lambda(\tau(m))$.
Thus applying \eqref{casimir} and Lemma \ref{lemconv}, we obtain constants
$C_{1,2}$ which are
independent of $m$ such that
\begin{align*}
&\nu(\Omega_K)=\left\|\lambda(\nu)+\rho_K\right\|^2-\left\|\rho_K\right\|^2\leq
\left\|\lambda(\nu')\right\|^2+C_1(1+\left\|\lambda(\nu')\right\|)\\
&\leq m^2\left(\sum_{j=1}^{n}\tau_j^2\right)+C_2m.
\end{align*}
One the other hand by \eqref{casimir} we have
\begin{align*}
\tau(m)(\Omega)=&\left\|\Lambda(\tau(m))+\rho_G\right\|^2-
\left\|\rho_G\right\|^2\\ =&\sum_{j=1}^{n+1}(m\tau_{j}+n+1-j)^2-\sum_{j=1}^{n+1}
(n+1-j)^2\geq m^2\sum_{j=1}^{n+1}\tau_j^2.
\end{align*}
This implies the proposition.
\end{proof}

\subsection{The case $G=\SL(3,\R)$.} 
We use the notation of section \ref{secdelta1}. We choose the Cartan subalgebra 
$\hL\subset\gL$, which is defined by \eqref{cartan}. The fundamental weights
$\omega_i\in\hL^*_\C$, $i=1,2$, are given by \eqref{fundwei1}.
Let $\Lambda\in\hL^*_\C$ be a highest weight. For $m\in\N$ let $\tau_\Lambda(m)$
be the irreducible 
representation with highest weight $m\Lambda$.
\begin{prop}\label{CasSLdrei}
Assume that $\Lambda$ satisfies $\Lambda_\theta\neq\Lambda$. Then there exists
$C_\Lambda>0$ such that 
\[
E_p(\tau_\Lambda(m))\ge
\frac{1}{9}m^2-C_\Lambda m
\]
for all $m\in\N$ and $p=0,\dots,5$.
\end{prop}
\begin{proof}
There exist
$\tau_1,\tau_2\in\N_0$ such that $\Lambda=\tau_1\omega_1+\tau_2\omega_2$.
Note that by \eqref{roots6} and \eqref{isom6} one has 
$\rho_G=f_1+f_2$. Then by \eqref{fundwei1} and \eqref{inProd} we get
\begin{align*}
\tau_\Lambda(m)(\Omega)=&\left\|m\Lambda+\rho_G\right\|^2-\left\|\rho_G
\right\|^2 \\ =&\frac{4(\tau_1^2+\tau_1\tau_2+\tau_2^2)}{9}m^2
+\frac{4(\tau_1+\tau_2)}{3}m.
\end{align*}
Next recall that there is a natural isomorphism $\kL_\C\cong 
{\mathfrak{su}}(2)_\C=\slL(2,\C)$ (see \cite[Sect. 4.9]{Ha}). Furthermore if we 
embed $\slL(2,\C)$ into $\gL_\C$ as an upper left block then $\tL_\C$ is 
isomorphic to a Cartan subalgebra of $\slL(2,\C)$. For $j\in\N$ we let $\nu_j$ 
denote the representation of $\kL_\C$ with highest weight $jf_2$. Then we 
deduce from the branching law from
$\GL_3(\C)$ to $\GL_2(\C)$, \cite[Theorem 8.1.1]{GW}
that 
\begin{align*}
\tau_\Lambda(m)|_{\kL_{\C}}=\bigoplus_{j=0}^{m\tau_1}\bigoplus_{k=0}^{m\tau_2}
\nu_{j+k}.
\end{align*}
If we use 
\begin{align*}
\nu_j(\Omega_K)=\frac{j^2}{3}+\frac{2}{3}j.
\end{align*}
and argue as in the proof of Proposition \ref{lowerbd5}, we obtain a constant
$C$ which is independent of $\tau_1,\tau_2$ and 
$m$  such that for every $\nu\in\hat{K}$ with
$\left[\nu_p(\tau(m)):\nu\right]\neq 0$ for some $p$ one has 
\begin{align*}
\nu(\Omega_K)\leq
\frac{\left(m(\tau_1+\tau_2)+C\right)^2}{3}+\frac{
2\left(m(\tau_1+\tau_2)+C\right)}{3}.
\end{align*}
Thus we obtain a constant $C_\Lambda$ such that for every $m$ and every $p$ one
has 
\begin{align*}
E_p(\tau_\Lambda(m))\ge \frac{(\tau_1-\tau_2)^2}{9}m^2-C_\Lambda m.
\end{align*}
By \eqref{Lambdatheta2} the condition $\Lambda_\theta\neq\Lambda$ is equivalent
to $\tau_1\neq \tau_2$. 
This proves the Proposition.
\end{proof}

Now we can summarize our results. 
\begin{prop}\label{lowerbd6}
Let $\delta(\widetilde X)=1$, $\widetilde{X}$ odd-dimensional. Let
$\lambda\in\hL_\C^*$ be a highest weight
with $\lambda_\theta\neq\lambda$. 
For $m\in\N$ let 
$\tau_\lambda(m)$ be the irreducible representation of $G$ with highest weight 
$m\lambda$. There exist $C_1,C_2>0$ such that
\[
E_p(\tau_\lambda(m))\ge C_1m^2-C_2
\]
for all $p=0,\dots,d$ and $m\in\N$.
\end{prop}
\begin{proof}
Let $\lambda=\lambda_0+\lambda_1$ with $\lambda_0\in\tL_{0,\C}^*$ and
$\lambda_1\in\hL_{1,\C}^*$ highest weights, and
$(\lambda_1)_\theta\neq\lambda_1$. 
Let $\tau_i(m)$, $i=0,1$, be the 
irreducible representations of $G_i$ with highest weight $m\lambda_i$. Then
$\tau(m)=\tau_0(m)\otimes\tau_1(m)$. 
Let $E_{0,p}(\tau_0(m))$ and $E_{1,p}(\tau_1(m))$ be defined by \eqref{endoi}.
By
Lemma \ref{genbd} there exists $C>0$ such that
\[
E_{0,p}(\tau_0(m))\ge -Cm
\]
for all $p=0,\dots,d$ and $m\in\N$. Furthermore, by Proposition \ref{lowerbd5} 
and Proposition \ref{CasSLdrei} there exist $C_3,C_4>0$ such that
\[
E_{1,p}(\tau_1(m))\ge C_3 m^2-C_4
\]
for all $p=0,\dots,d$ and $m\in\N$. Combined with \eqref{decomp3} the proof
follows.
\end{proof}
\begin{corollary}\label{lowerbd7}
Let the assumptions be as in Proposition $\mathrm{\ref{lowerbd6}}$. There exist 
constants $C_1,C_2>0$ such that
\[
\Delta_p(\tau_\lambda(m))\ge C_1m^2-C_2
\]
for all $p=0,\dots,d$ and $m\in\N$.
\end{corollary}
\begin{proof}
Recall that the Bochner-Laplace operator satisfies $\Delta_{\nu_p(\tau(m))}\ge
0$. 
Hence the corollary follows from \eqref{bochhodge1} and Proposition 
\ref{lowerbd6}.
\end{proof}

\section{Proof of the main results}\label{secmainres}
\setcounter{equation}{0}
First assume that $\delta(\widetilde X)\neq0$. Note that 
$\delta(\widetilde X)=0$ implies that $\dim\widetilde X$ is even. Hence, it
follows from Proposition \ref{vanish1} that $T_X(\tau)=1$ for all 
finite-dimensional irreducible representations of $G$, which proves part (i) of
Theorem \ref{th-main1}.

Now assume that $\delta(\widetilde X)=1$. Let $\hL\subset \gL$ be a fundamental
Cartan subalgebra. 
Let $\lambda\in\hL_\C^*$ be a highest weight with
$\lambda_\theta\neq\lambda$. For $m\in\N$ let $\tau(m)$ be the irreducible
representation of $G$ with highest weight $m\lambda$. Then $\tau(m)\not\cong
\tau(m)_\theta$ for all $m\in\N$. Hence by \cite[Chapter VII, Theorem 6.7]{BW}
we have 
$H^p(X,E_{\tau(m)})=0$ for all $p=0,\dots,d$. Then by \eqref{anator2} we have
\begin{equation}\label{anator5}
\log T_X(\tau(m))=\frac{1}{2}\frac{d}{ds}\left(\frac{1}{\Gamma(s)}
\int_0^\infty t^{s-1}K(t,\tau(m))\,dt\right)\bigg|_{s=0}.
\end{equation}
Since $\tau(m)$ is acyclic and $\dim X$ is odd, $T_X(\tau(m))$ is metric
independent \cite[Corollary 2.7]{Mu2}. 
Especially we can rescale the metric by $\sqrt{m}$ without changing 
$T_X(\tau(m))$. Equivalently we can replace $\Delta_p(\tau(m))$ by $\frac{1}{m}
\Delta_p(\tau(m))$. Using \eqref{anator5} we get
\[
\log T_X(\tau(m))=\frac{1}{2}\frac{d}{ds}\left(\frac{1}{\Gamma(s)}
\int_0^\infty t^{s-1}K\left(\frac{t}{m},\tau(m)\right)\,dt\right)\bigg|_{s=0}.
\]
To continue,  we split the $t$-integral into the integral over $[0,1]$ and
the integral over $[1,\infty)$. This leads to
\begin{equation}\label{splitint}
\begin{split}
\log T_X(\tau(m))=&\frac{1}{2}\frac{d}{ds}\left(\frac{1}{\Gamma(s)}
\int_0^1 t^{s-1}K\left(\frac{t}{m},\tau(m)\right)\,dt\right)\bigg|_{s=0}\\
&+\frac{1}{2}\int_1^\infty t^{-1}K\left(\frac{t}{m},\tau(m)\right)\,dt.
\end{split}
\end{equation}
We first consider the second term on the right hand side. To this end we need
the following lemma.
\begin{lem}\label{Katotr}
Let $h_t^{\tau(m),\:p}$ be defined by \eqref{Defh} and let $H_t^0$ be the heat
kernel of the Laplacian $\widetilde{\Delta}_0$ on $C^\infty(\widetilde{X})$.
There exist $m_0\in\N$ and $C>0$ such that for all $m\geq m_0$, $g\in G$, 
$t\in(0,\infty)$ and $p\in\{0,\dots,d\}$ one has
\begin{align*}
\left|h_t^{\tau(m),\:p}(g)\right|\leq
C\dim(\tau(m))e^{-t\frac{m^2}{2}} H_t^0(g).
\end{align*}
\end{lem}
\begin{proof}
Let $p\in\{0,\dots,n\}$. 
Let $H_t^{\nu_p(\tau(m))}$ be the kernel of
$e^{-t\widetilde{\Delta}_{\nu_p(\tau(m))}}$
and let $H_t^{\tau(m),\:p}$ be the kernel of
$e^{-t\widetilde{\Delta}_p(\tau(m))}$. 
By \eqref{heatfact} we have
\begin{align*}
H_t^{\tau(m),\:p}(g)=e^{-tE_p(\tau(m))}\circ H_t^{\nu_p(\tau(m))}(g).
\end{align*}
Thus by proposition \ref{Kato} and Proposition \ref{lowerbd6} there exists
an $m_0$ such that for $m\geq m_0$ one has
\begin{equation}\label{estim6}
\left\|H_t^{\tau(m),\:p}(g)\right\|\leq
e^{-t\frac{m^2}{2}}H_t^0(g).
\end{equation}
Taking the trace in $\End(\Lambda^p\mathfrak{p}^*\otimes V_{\tau(m)})$ for
every 
$p\in\{0,\dots,d\}$, the lemma follows. 
\end{proof}
Using \eqref{anator3}, \eqref{alter} and Lemma \ref{Katotr}, we obtain
\[
\begin{split}
\left|K\left(\frac{t}{m},\tau(m)\right)\right|
&\leq C
e^{-\frac{m}{2}t}\dim(\tau(m))\int_{\Gamma\bs G}\sum_{\gamma\in\Gamma}
H_{t/m}^0(g^{-1}\gamma g)\,d\dot g\\
&= C e^{-\frac{m}{2}t}\dim(\tau(m))\Tr(e^{-\frac{t}{m}\Delta_0}).
\end{split}
\]
Furthermore, by the heat asymptotic \cite{Gi} we have 
\[
\Tr(e^{-\frac{1}{m}\Delta_0})=C_d\vol(X)m^{d/2}+O\left(m^{(d-1)/2}\right)
\]
as $m\to\infty$. Hence there exists $C_1>0$ such that
\[
\left|K\left(\frac{t}{m},\tau(m)\right)\right|\le C_1 m^{d/2}
\dim(\tau(m))e^{-\frac{m}{2}t},\quad t\ge 1.
\]
Thus we obtain
\[
\left|\int_1^\infty t^{-1}K\left(\frac{t}{m},\tau(m)\right)\;dt\right|\le
C_1 m^{d/2}\dim(\tau(m))e^{-m/4}\int_1^\infty t^{-1}e^{-\frac{m}{4}t}\;dt.
\]
Using Weyl's dimension formula, it follows that
\begin{equation}\label{term2}
\int_1^\infty
t^{-1}K\left(\frac{t}{m},\tau(m)\right)\;dt=O\left(e^{-m/8}\right).
\end{equation}
Now we turn to the first term on the right hand side of \eqref{splitint}. We
need to estimate $K(t,\tau(m))$ for $0<t\le 1$. 
To this end we use \eqref{anator3} to decompose $K(t,\tau(m))$ into the sum
of two terms: The contribution of the identity
\begin{equation}\label{contribid}
I(t,\tau(m)):=\vol(X)k_t^{\tau(m)}(1),
\end{equation}
where $k_t^{\tau(m)}$ is defined by \eqref{alter}, and the remaining term
\begin{align*}
H(t,\tau(m)):=\int_{\Gamma\bs G}\sum_{\substack{\gamma\in\Gamma\\ \gamma\neq
1}}k_t^{\tau(m)}(g^{-1}\gamma g)\,d\dot g
\end{align*}
First we consider $H(t,\tau(m))$.  
Using Proposition \ref{Katotr} and Proposition \ref{esthyp}, it follows that
for every $m\geq m_0$ and every $t\in\left(0,1\right]$ we have
\begin{align*}
\sum_{\substack{\gamma\in\Gamma\\ \gamma\neq
1}}\left|k_{t}^{\tau(m)}(g^{-1}\gamma g)\right|
&\leq C e^{-t\frac{m^2}{2}}\dim(\tau(m))
\sum_{\substack{\gamma\in\Gamma\\ \gamma\neq
1}}H_{t}^{0}(g^{-1}\gamma g)\\ &\leq C_1 \dim(\tau(m))
e^{-t\frac{m^2}{2}}e^{-c_0/t}.
\end{align*}
Hence using Weyl's dimension formula we get
\begin{equation*}
\left| H\left(\frac{t}{m},\tau(m)\right)\right|\le C_2 e^{-c_1m}e^{-c_1/t},
\quad 0<t\le 1.
\end{equation*}
This implies that there is $c_2>0$ such that
\begin{equation}\label{term3}
\begin{split}
&\frac{d}{ds}\left(\frac{1}{\Gamma(s)}\int_0^1 t^{s-1}H\left(\frac{t}{m},\tau(m)
\right)\,dt\right)\bigg|_{s=0} \\ &=\int_0^1 t^{-1}H\left(\frac{t}{m},\tau(m)
\right)\,dt=O\left(e^{-c_2m}\right)
\end{split}
\end{equation}
as $m\to\infty$.

It remains to consider the contribution of the identity $I(t,\tau(m))$. By
Lemma \ref{Katotr} there exists $C>0$ such that for all $m\ge m_0$ and
$p=0,\dots,d$ we have
\[
|h_t^{\tau(m),p}(1)|\le C \dim(\tau(m)) e^{-t\frac{m^2}{2}}H_t^0(1).
\]
Next we estimate $H_t^0(1)$ using the Plancherel-Theorem. Since the function
$H_t^0(1)$ is $K$-biinvariant, the Plancherel-Theorem for
$H_t^0(1)$ reduces to the spherical Plancherel theorem \cite[Theorem 7.5]{He}.
Thus if $Q=MAN$ is 
a fixed minimal standard parabolic subgroup, it follows from \eqref{InfchHS}
that
\begin{align*}
H_t^0(1)=e^{-t\left\|\rho_\aL\right\|^2}\int_{\aL^*}
e^{-t\left\|\nu\right\|^2}\beta(\nu)d\nu,
\end{align*}
where $\beta(\nu)$ is the spherical Plancherel-density.
Thus there  exists $C_1>0$ such that $|H_t^0(1)|\le C_1$
for $t\ge 1$. Hence, by \eqref{alter} we get
\[
|k_t^{\tau(m)}(1)|\le C_2\dim(\tau(m))e^{-t\frac{m^2}{2}}
\]
for $t\ge 1$ and $m\ge m_0$. By \eqref{contribid} and Weyl's dimension formula
 it follows that there exist $C,c>0$ such that
\begin{equation}
\Bigg|I\left(\frac{t}{m},\tau(m)\right)\Bigg|\le C e^{-cm} e^{-ct}
\end{equation}
for $t\ge 1$ and $m\ge m_0$. Hence we get
\begin{equation}
\begin{split}
&\frac{d}{ds}\left(\frac{1}{\Gamma(s)}\int_0^1t^{s-1}I\left(\frac{t}{m},\tau(m)
\right)\,dt\right)\Bigg|_{s=0} \\ &= \frac{d}{ds}\left(\frac{1}{\Gamma(s)}
\int_0^\infty t^{s-1}I\left(\frac{t}{m},\tau(m)\right)\,dt\right)\Bigg|_{s=0}
+O\left(e^{-cm}\right)
\end{split}
\end{equation}
for $m\ge m_0$. Since we are assuming that $\delta(\widetilde X)=1$, 
$\dim(X)$ is odd.
Then it follows from \eqref{asympexp3} and the definition of $k_t^{\tau(m)}$
by \eqref{alter}, that $k_t^{\tau(m)}(1)$ has an asymptotic expansion of the
form
\[
k_t^{\tau(m)}(1)\sim\sum_{j=0}^\infty c_j t^{-d/2+j}.
\]
Since $d$ is odd, the expansion has no constant term. This implies that
\[
\int_0^\infty t^{s-1} I(t,\tau(m))\;dt
\]
is holomorphic at $s=0$. Therefore we get
\begin{equation*}
\begin{split}
&\frac{d}{ds}\left(\frac{1}{\Gamma(s)}
\int_0^\infty t^{s-1}I\left(\frac{t}{m},\tau(m)\right)\,dt\right)\Bigg|_{s=0}
 \\ =&\frac{d}{ds}\left(\frac{1}{\Gamma(s)}
\int_0^\infty t^{s-1}I\left(t,\tau(m)\right)\,dt\right)\Bigg|_{s=0}.
\end{split}
\end{equation*}
By definition, the right hand side equals  $\log T_X^{(2)}(\tau(m))$, 
$T_X^{(2)}(\tau(m))$ the $L^2$-torsion. Combined
with \eqref{splitint}, \eqref{term2} and \eqref{term3} we obtain
\begin{equation}\label{torequ}
\log T_X(\tau(m))=\log T_X^{(2)}(\tau(m))+O\left(e^{-cm}\right).
\end{equation}
This proves Proposition \ref{prop-l2tor1}. 
\hfill \endproof

Combining Proposition \ref{l2torprod} with Proposition \ref{L2spin} and
Proposition \ref{L2SL3}, we obtain Proposition \ref{prop-l2tor2}. 
Together with Proposition \ref{prop-l2tor1} we obtain part (ii) of
Theorem \ref{th-main1}.

Corollary \ref{cor-spin} follows from Proposition \ref{L2spin} and
Corollary \ref{cor-sl3} follows from Proposition \ref{L2SL3}.

\end{document}